\documentclass{article}%
\usepackage[mathscr]{eucal}
\usepackage[super]{nth}
\usepackage{amsmath}
\usepackage{amssymb}
\usepackage{amsthm}
\usepackage[latin1]{inputenc}
\usepackage{pxfonts}
\usepackage{txfonts}
\usepackage{graphicx}
\usepackage[usenames,dvipsnames]{pstricks}
\usepackage{epsfig}
\usepackage{pst-grad}
\usepackage{pst-plot}
\usepackage{ifthen}
\usepackage{color}
\usepackage{amsfonts}
\usepackage[usenames,dvipsnames]{pstricks}
\usepackage{epsfig}
\usepackage{pst-grad}
\usepackage{pst-plot}%
\providecommand{\U}[1]{\protect\rule{.1in}{.1in}}
\newtheorem{theorem}{Theorem}[section]

\newtheorem{proposition}[theorem]{Proposition}
\newtheorem{lemma}[theorem]{Lemma}

\theoremstyle{definition}
\newtheorem{problem}{Problem}

\newtheorem{definition}[theorem]{Definition}

\begin{document}

\title{Regularity principle in sequence spaces \\and applications}
\author{D. Pellegrino, \quad J. Santos, \quad D. Serrano-Rodr\'iguez, \quad E. Teixeira}
\date{}
\maketitle

\begin{abstract}
We prove a nonlinear regularity principle in sequence spaces which produces
universal estimates for special series defined therein. Some consequences are
obtained and, in particular, we establish new inclusion theorems for multiple summing operators. Of independent interest, we settle all Grothendieck's type $(\ell_{1},\ell_{2})$ theorems for multilinear operators. We further employ the new regularity principle to solve the classification
problem concerning all pairs of admissible exponents in the anisotropic Hardy--Littlewood inequality.

\end{abstract}
\tableofcontents



\section{Introduction}

Regularity arguments are fundamental tools in the analysis of a variety of
problems as they often pave the way to important discoveries in the realm of
mathematics and its applications. Regularity results may appear in many
different configurations, sometimes quite explicitly as in the theory of
diffusive PDEs, sometimes in a more subtle form, and in this article we are
interested in the following universality problem, which drifts a hidden
regularity principle in it:


\begin{problem}
\label{Univ. Problem} Let $p\geq1$ be a real number, $X,Y,W_{1},W_{2}$ be
non-void sets, $Z_{1},Z_{2},Z_{3}$ be normed spaces and $f\colon X\times
Y\rightarrow Z_{1},\ g\colon X\times W_{1}\rightarrow Z_{2},\ h\colon Y\times
W_{2}\rightarrow Z_{3}$ be particular maps. Assume there is a constant $C>0$
such that%
\begin{equation}
{%
{\textstyle\sum\limits_{i=1}^{m_{1}}}
{\textstyle\sum\limits_{j=1}^{m_{2}}}
\left\Vert f(x_{i},y_{j})\right\Vert ^{p}\leq C\left(  \sup_{w\in W_{1}}%
{\textstyle\sum\limits_{i=1}^{m_{1}}}
\left\Vert g(x_{i},w)\right\Vert ^{p}\right)  \cdot\left(  \sup_{w\in W_{2}}%
{\textstyle\sum\limits_{j=1}^{m_{2}}}
\left\Vert h(y_{j},w)\right\Vert ^{p}\right)  }, \label{ree}%
\end{equation}
for all $x_{i}\in X$, $y_{j}\in Y$ and $m_{1},m_{2}\in\mathbb{N}$. Are there
(universal) positive constants $\epsilon\sim\delta$, and $\tilde{C}%
_{\delta,\epsilon}$ such that%
\begin{equation}
\text{ }\left(  {%
{\textstyle\sum\limits_{i=1}^{m_{1}}}
{\textstyle\sum\limits_{j=1}^{m_{2}}}
}\left\Vert f(x_{i},y_{j})\right\Vert ^{p+\delta}\right)  ^{\frac{1}{p+\delta
}}\leq\tilde{C}_{\delta,\epsilon}\cdot\left(  \sup_{w\in W_{1}}%
{\textstyle\sum\limits_{i=1}^{m_{1}}}
\left\Vert g(x_{i},w)\right\Vert ^{p+\epsilon}\right)  ^{\frac{1}{p+\epsilon}%
}\left(  \sup_{w\in W_{2}}%
{\textstyle\sum\limits_{j=1}^{m_{2}}}
\left\Vert h(y_{j},w)\right\Vert ^{p+\epsilon}\right)  ^{\frac{1}{p+\epsilon}%
}, \label{ree1}%
\end{equation}
for all $x_{i}\in X$, $y_{j}\in Y$ and $m_{1},m_{2}\in\mathbb{N}$?
\end{problem}

It turns out that many classical questions in mathematical analysis,
permeating several different fields, can be framed into the formalism of the
universality Problem \ref{Univ. Problem}. A key observation is that the
existence of a leeway, $\epsilon>0$, of an increment $\delta>0$, and of a
corresponding bound $\tilde{C}_{\delta,\epsilon}>0$ bears a regularity
principle for the orderly problem which often reveals important aspects of the
theory underneath.

In this work, under appropriate assumptions, we solve the universality problem
in a very general setting. This is a flexible, effective tool and we apply it
in the investigation of two central problems in mathematical analysis, namely
inclusion type theorems for summing operators and the solution of
the classification problem in sharp anisotropic Hardy--Littlewood inequality.

The theory of absolutely summing operators plays an important role in the
study of Banach Spaces and Operator Theory, with deep inroads in other areas
of Analysis. Grothendieck's inequality, described by Grothendieck as
\textquotedblleft the fundamental theorem in the metric theory of tensor
products\textquotedblright\ is one of the cornerstones of the theory of
absolutely summing operators, and a fundamental general result in Mathematics
\cite{RR,GG,PISIER}. For linear operators, $p$-summability implies
$q$-summability whenever $1\leq p\leq q.$ More generally, if $1\leq p_{j}\leq
q_{j},$ $j=1,2,$ every absolutely $\left(  p_{1};p_{2}\right)  $-summing
operators is absolutely $\left(  q_{1};q_{2}\right)  $-summing whenever
\[
\frac{1}{p_{2}}-\frac{1}{p_{1}}\leq\frac{1}{q_{2}}-\frac{1}{q_{1}}.
\]
Results of this sort are usually called \textquotedblleft inclusion
results\textquotedblright. In the multilinear setting inclusion results are
more intriguing. For instance, every multiple $p$-summing multilinear operator
is multiple $q$-summing whenever $1\leq p\leq q\leq2$, but this is not valid
beyond the threshold $2.$ Our first application of the regularity principle
provides new inclusion theorems for multiple summing operators overtaking the
barrier $2$. Our proof is based on delicate inclusion properties that follow
as consequence of the general regularity principle we will establish.

The second featured application we carry on pertains to the theory of
anisotropic Hardy-Littlewood inequality. Given numbers $p,q\in\lbrack
2,\infty]$ and a pair of exponents $(a,b)$, one is interested in the existence
of a universal constant $C_{p,q,a,b}\geq1$ such that
\begin{equation}
\left(  \sum_{i=1}^{n}\left(  \sum_{j=1}^{n}\left\vert T(e_{i},e_{j}%
)\right\vert ^{a}\right)  ^{\frac{1}{a}\cdot b}\right)  ^{\frac{1}{b}}\leq
C_{p,q,a,b}\cdot\left\Vert T\right\Vert , \label{inicio}%
\end{equation}
for all bilinear operators $T\colon\ell_{p}^{n}\times\ell_{q}^{n}%
\rightarrow\mathbb{K}$ and all positive integers $n$; {here and henceforth
$\mathbb{K}$ denotes the field of real or complex scalars.} Questions of this
sort are essential in many areas of mathematical analysis and dates back, at
least, to the works of Toeplitz \cite{toe} and Riesz \cite{riesz}. Hardy and
Littlewood, in \cite{hardy}, establish the existence of particular
anisotropic exponents for which \eqref{inicio} holds and since then a key
issue in the theory has been to investigate the optimal range of anisotropic
exponents. As an application of the regularity principle, we obtain a complete
classification of all pairs of anisotropic exponents $(a,b)$ for which
estimate \eqref{inicio} holds, providing henceforth a definitive solution to
the problem. We show that \eqref{inicio} is verified if, and only if, the pair
of anisotropic exponents $(a,b)$ lies in $[\frac{q}{q-1},\infty)\times
\lbrack\frac{pq}{pq-p-q},\infty)$ and verifies
\[
\frac{1}{a}+\frac{1}{b}\leq\frac{3}{2}-\left(  \frac{1}{p}+\frac{1}{q}\right)
.
\]
In the case \eqref{inicio} fails to hold we obtain the precise dimension
blow-up rate. We further comment on generalizations of such results to the
multilinear setting.

The paper is organized as follows. In Section 2 we introduce and prove the
regularity principle --- the \textit{tour of force} of this work. In Section
3, we explore the regularity principle as to establish new inclusion
properties for multiple summing operators. In Section 4 we prove an all-embracing Grothendieck's type $(\ell_1, \ell_2)$ theorem.  Section 5 is devoted to the
solution of the anisotropic Hardy--Littlewood inequality problem for bilinear
forms. We show that an application of the regularity principle classifies all
the admissible anisotropic exponents for which Hardy--Littlewood inequality is
valid. In Section 5 we determine the exact blow-up rate for non-admissible
Hardy--Littlewood exponents, as dimension goes to infinity. In the final section we discuss some new insights
concerning Hardy--Littlewood inequality in the multilinear setting, which may
pave the way to further investigations in the theory.

\section{The Regularity Principle}

\label{Sct RP}

In this Section we will establish a nonlinear regularity principle which
greatly expands the investigation initiated in \cite{advances} concerning
inclusion properties for sums in one index. Regularity results for summability
in multiple indexes, objective of our current study, are rather more
challenging and involve a number of new technical difficulties. Accordantly,
it is indeed a rather more powerful analytical tool and we shall explore its
full strength in the upcoming sections.

Let $\,Z_{1},\,V$ and $W_{1},\,W_{2}$ be arbitrary non-void sets and $Z_{2}$
be a vector space. For $t=1,2,$ let\
\begin{align*}
R_{t}\colon Z_{t}\times W_{t}  &  \longrightarrow\lbrack0,\infty),\text{ and
}\\
S\colon Z_{1}\times Z_{2}\times V  &  \longrightarrow\lbrack0,\infty)
\end{align*}
be arbitrary mappings satisfying%

\begin{align*}
R_{2}\left(  \lambda z,w\right)   &  =\lambda R_{2}\left(  z,w\right)  ,\\
S\left(  z_{1},\lambda z_{2},v\right)   &  =\lambda S\left(  z_{1}%
,z_{2},v\right)
\end{align*}
for all real scalars $\lambda\geq0.$ In addition, all along the paper we adopt the convention $\frac{1}%
{0}=\infty$, $\frac{1}{\infty}=0$, and throughout this Section we
always work in the range $p_{1}\geq1$, and assume%

\begin{equation}
\label{R is bdd}\left(  \sup_{w\in W_{t}}\sum_{j=1}^{m_{t}}R_{t}\left(
z_{t,j},w\right)  ^{p_{1}}\right)  ^{\frac{1}{p_{1}}}<\infty, \quad t=1,2.
\end{equation}

Note these are rather general, weak hypotheses on the governing maps $S,
R_{1}, R_{2}$; in particular no continuity conditions are imposed.

\begin{theorem}
[Regularity Principle]\label{RP} Let $1\leq p_{1}\leq p_{2}<2p_{1}$ and
assume
\[
\left(  \sup_{v\in V}\sum_{i=1}^{m_{1}}\sum_{j=1}^{m_{2}}S(z_{1,i}%
,z_{2,j},v)^{p_{1}}\right)  ^{\frac{1}{p_{1}}}\hspace{-0.2cm}\leq C\left(
\sup_{w\in W_{1}}\sum_{i=1}^{m_{1}}R_{1}\left(  z_{1,i},w\right)  ^{p_{1}%
}\right)  ^{\frac{1}{p_{1}}}\left(  \sup_{w\in W_{2}}\sum_{j=1}^{m_{2}}%
R_{2}\left(  z_{2,j},w\right)  ^{p_{1}}\right)  ^{\frac{1}{p_{1}}},
\]
for all $z_{1,i}\in Z_{1},z_{2,j}\in Z_{2},$ all $i=1,...,m_{1}$ and
$j=1,...,m_{2}$ and $m_{1},m_{2}\in\mathbb{N}$. Then%
\[
\left(  \sup_{v\in V}\sum_{i=1}^{m_{1}}\sum_{j=1}^{m_{2}}S(z_{1,i}%
,z_{2,j},v)^{\frac{p_{1}p_{2}}{2p_{1}-p_{2}}}\right)  ^{\frac{2p_{1}-p_{2}%
}{p_{1}p_{2}}}\hspace{-0.3cm}\leq C\left(  \sup_{w\in W_{1}}\sum_{i=1}^{m_{1}%
}R_{1}\left(  z_{1,i},w\right)  ^{p_{2}}\right)  ^{\frac{1}{p_{2}}}\left(
\sup_{w\in W_{2}}\sum_{j=1}^{m_{2}}R_{2}\left(  z_{2,j},w\right)  ^{p_{2}%
}\right)  ^{\frac{1}{p_{2}}},
\]
for all $z_{1,i}\in Z_{1},z_{2,j}\in Z_{2},$ all $i=1,...,m_{1}$ and
$j=1,...,m_{2}$ and $m_{1},m_{2}\in\mathbb{N}$.
\end{theorem}

\begin{proof}
{Consider $\left(  z_{1,i}\right)  _{i=1}^{m_{1}}$ in $Z_{1}$ and $\left(
z_{2,j}\right)  _{j=1}^{m_{2}}$ in $Z_{2}$ and define, for all $z\in Z_{1}$
and $v\in V$,}
\[
S_{1}(z,v)=\left(  \sum_{j=1}^{m_{2}}S(z,z_{2,j},v)^{p_{1}}\right)  ^{1/p_{1}%
}.
\]
We estimate%
\begin{align*}
\sup_{v\in V}\sum_{i=1}^{m_{1}}S_{1}(z_{1,i},v)^{p_{1}} &  =\sup_{v\in V}%
\sum_{i=1}^{m_{1}}\sum_{j=1}^{m_{2}}S(z_{1,i},z_{2,j},v)^{p_{1}}\\
&  \leq C^{p_{1}}\left(  \sup_{w\in W_{1}}\sum_{i=1}^{m_{1}}R_{1}\left(
z_{1,i},w\right)  ^{p_{1}}\right)  \left(  \sup_{w\in W_{2}}\sum_{j=1}^{m_{2}%
}R_{2}\left(  z_{2,j},w\right)  ^{p_{1}}\right)  \\
&  =C_{1}\sup_{w\in W_{1}}\sum_{i=1}^{m_{1}}R_{1}\left(  z_{1,i},w\right)
^{p_{1}}%
\end{align*}
with%
\[
C_{1}=C^{p_{1}}\sup_{w\in W_{2}}\sum_{j=1}^{m_{2}}R_{2}\left(  z_{2,j}%
,w\right)  ^{p_{1}}.
\]
{We can consider a new sequence in $Z_{1}$ where each term is repeated with a
prescribed frequency. Let $\eta_{i}$ be the number of times each $z_{1,i}$
appears respectively. We have
\begin{equation}
\sum_{i=1}^{m_{1}}\eta_{i}S_{1}(z_{1,i},v)^{p_{1}}\leq C_{1}%
\sup_{w\in W_{1}}\sum_{i=1}^{m_{1}}\eta_{i}R_{1}\left(  z_{1,i},w\right)
^{p_{1}},\label{yr2}%
\end{equation}
for all $z_{1,1},\ldots,z_{1,m_{1}}\in Z_{1}$ and all $v\in V$. Now, passing from integers to
rationals by \textquotedblleft cleaning\textquotedblright\ denominators and
from rationals to real numbers using density, we conclude that (\ref{yr2})
holds for positive real numbers $\eta_{1},\cdots,\eta_{m_{1}}.$ Define
\[
\frac{1}{p}=\frac{1}{p_{1}}-\frac{1}{p_{2}}.
\]
For each $i=1,\cdots,m_{1}$, consider the map $\lambda_{i}\colon
V\rightarrow\lbrack0,\infty)$ by
\[
\lambda_{i}(v):=S_{1}(z_{1,i},v)^{\frac{p_{2}}{p}}.
\]
Hence, we readily have
\begin{align*}
\lambda_{i}(v)^{p_{1}}S_{1}(z_{1,i},v)^{p_{1}} &  =S_{1}(z_{1,i}%
,v)^{\frac{p_{1}p_{2}}{p}}S_{1}(z_{1,i},v)^{p_{1}}\\
&  =S_{1}(z_{1,i},v)^{p_{2}}.
\end{align*}
Recalling that (\ref{yr2}) is valid for arbitrary positive real numbers
$\eta_{i},$ we get, for $\eta_{i}=\lambda_{i}(v)^{p_{1}}$,
\begin{align*}
\sum_{i=1}^{m_{1}}S_{1}(z_{1,i},v)^{p_{2}} &  =\sum_{i=1}^{m_{1}}\lambda
_{i}(v)^{p_{1}}S_{1}(z_{1,i},v)^{p_{1}}\\
&  \leq C_{1}\sup_{w\in W_{1}}\sum_{i=1}^{m_{1}}\lambda_{i}(v)^{p_{1}}%
R_{1}\left(  z_{1,i},w\right)  ^{p_{1}},
\end{align*}
for every $v\in V$. Also, taking into account the relation
\[
\frac{1}{(p/p_{1})}+\frac{1}{(p_{2}/p_{1})}=1,
\]
and H\"older's inequality we obtain
\begin{align*}
\sum_{i=1}^{m_{1}}S_{1}(z_{1,i},v)^{p_{2}} &  \leq C_{1}\sup_{w\in W_{1}}%
\sum_{i=1}^{m_{1}}\lambda_{i}(v)^{p_{1}}R_{1}\left(  z_{1,i},w\right)
^{p_{1}}\\
&  \leq C_{1}\sup_{w\in W_{1}}\left[  \left(  \sum_{i=1}^{m_{1}}\lambda
_{i}(v)^{\frac{p_{1}p}{p_{1}}}\right)  ^{\frac{p_{1}}{p}}\left(  \sum
_{i=1}^{m_{1}}R_{1}\left(  z_{1,i},w\right)  ^{p_{1}\frac{p_{2}}{p_{1}}%
}\right)  ^{\frac{p_{1}}{p_{2}}}\right]  \\
&  =C_{1}\left(  \sum_{i=1}^{m_{1}}\lambda_{i}(v)^{p}\right)  ^{\frac{p_{1}%
}{p}}\sup_{w\in W_{1}}\left(  \sum_{i=1}^{m_{1}}R_{1}\left(  z_{1,i},w\right)
^{p_{2}}\right)  ^{\frac{p_{1}}{p_{2}}}\\
&  =C_{1}\left(  \sum_{i=1}^{m_{1}}S_{1}(z_{1,i},v)^{p_{2}}\right)
^{\frac{p_{1}}{p}}\sup_{w\in W_{1}}\left(  \sum_{i=1}^{m_{1}}R_{1}\left(
z_{1,i},w\right)  ^{p_{2}}\right)  ^{\frac{p_{1}}{p_{2}}}.
\end{align*}
Therefore%
\[
\left(  \sum_{i=1}^{m_{1}}S_{1}(z_{1,i},v)^{p_{2}}\right)  ^{1-\frac{p_{1}}%
{p}}\leq C_{1}\sup_{w\in W_{1}}\left(  \sum_{i=1}^{m_{1}}R_{1}\left(
z_{1,i},w\right)  ^{p_{2}}\right)  ^{\frac{p_{1}}{p_{2}}}%
\]
for every $v\in V,$ and we can finally conclude that
\[
\left(  \sup_{v\in V}\sum_{i=1}^{m_{1}}S_{1}(z_{1,i},v)^{p_{2}}\right)
^{\frac{p_{1}}{p_{2}}}\leq C_{1}\sup_{w\in W_{1}}\left(  \sum_{i=1}^{m_{1}%
}R_{1}\left(  z_{1,i},w\right)  ^{p_{2}}\right)  ^{\frac{p_{1}}{p_{2}}}.
\]
Hence%
\begin{align*}
\left(  \sup_{v\in V}\sum_{i=1}^{m_{1}}S_{1}(z_{1,i},v)^{p_{2}}\right)
^{\frac{1}{p_{2}}} &  \leq C_{1}^{1/p_{1}}\sup_{w\in W_{1}}\left(  \sum
_{i=1}^{m_{1}}R_{1}\left(  z_{1,i},w\right)  ^{p_{2}}\right)  ^{\frac{1}%
{p_{2}}}\\
&  =C\left(  \sup_{w\in W_{2}}\sum_{j=1}^{m_{2}}R_{2}\left(  z_{2,j},w\right)
^{p_{1}}\right)  ^{\frac{1}{p_{1}}}\sup_{w\in W_{1}}\left(  \sum_{i=1}^{m_{1}%
}R_{1}\left(  z_{1,i},w\right)  ^{p_{2}}\right)  ^{\frac{1}{p_{2}}}.
\end{align*}
Recalling the definition of $S_{1}$ we have%
\begin{align*}
&  \left(  \sup_{v\in V}\sum_{i=1}^{m_{1}}\left(  \sum_{j=1}^{m_{2}}%
S(z_{1,i},z_{2,j},v)^{p_{1}}\right)  ^{p_{2}/p_{1}}\right)  ^{\frac{1}{p_{2}}%
}\\
&  \leq C\left(  \sup_{w\in W_{2}}\sum_{j=1}^{m_{2}}R_{2}\left(
z_{2,j},w\right)  ^{p_{1}}\right)  ^{\frac{1}{p_{1}}}\sup_{w\in W_{1}}\left(
\sum_{i=1}^{m_{1}}R_{1}\left(  z_{1,i},w\right)  ^{p_{2}}\right)  ^{\frac
{1}{p_{2}}}.
\end{align*}
Since $p_{1}\leq p_{2}$ we have%
\begin{align*}
&  \left(  \sup_{v\in V}\sum_{i=1}^{m_{1}}\sum_{j=1}^{m_{2}}S(z_{1,i}%
,z_{2,j},v)^{p_{2}}\right)  ^{\frac{1}{p_{2}}}\\
&  \leq C\left(  \sup_{w\in W_{2}}\sum_{j=1}^{m_{2}}R_{2}\left(
z_{2,j},w\right)  ^{p_{1}}\right)  ^{\frac{1}{p_{1}}}\sup_{w\in W_{1}}\left(
\sum_{i=1}^{m_{1}}R_{1}\left(  z_{1,i},w\right)  ^{p_{2}}\right)  ^{\frac
{1}{p_{2}}}.
\end{align*}
Now we look at the above inequality as%
\[
\left(  \sup_{v\in V}\sum_{j=1}^{m_{2}}S_{2}\left(  z_{2,j},v\right)  ^{p_{2}%
}\right)  ^{\frac{1}{p_{2}}}\leq C_{2}\left(  \sup_{w\in W_{2}}\sum
_{j=1}^{m_{2}}R_{2}\left(  z_{2,j},w\right)  ^{p_{1}}\right)  ^{\frac{1}%
{p_{1}}}%
\]
with
\begin{align*}
S_{2}\left(  z,v\right)   &  =\left(  \sum_{i=1}^{m_{1}}S(z_{1,i},z,v)^{p_{2}%
}\right)  ^{\frac{1}{p_{2}}},\\
C_{2} &  =C\sup_{w\in W_{1}}\left(  \sum_{i=1}^{m_{1}}R_{1}\left(
z_{1,i},w\right)  ^{p_{2}}\right)  ^{\frac{1}{p_{2}}}%
\end{align*}
for every }$z\in Z_{2}$ and $v\in V.$ {Since
\begin{align*}
S_{2}\left(  \lambda z,v\right)   &  = \lambda
S_{2}\left(  z,v\right)  ,\\
R_{2}\left(  \lambda z,w\right)   &  = \lambda
R_{2}\left(  z,w\right)  ,
\end{align*}
for all non negative scalars $\lambda$, we can use a somewhat similar argument. Recall that
\[
\frac{1}{p}=\frac{1}{p_{1}}-\frac{1}{p_{2}}%
\]
and note that%
\[
\frac{1}{p}=\frac{1}{p_{2}}-\frac{1}{q},
\]
with%
\[
q=\frac{p_{1}p_{2}}{2p_{1}-p_{2}}.
\]
For each $j=1,\cdots,m_{2}$, consider the map $\vartheta_{j}\colon
V\rightarrow\lbrack0,\infty)$ given by
\[
\vartheta_{j}(v):=S_{2}(z_{2,j},v)^{\frac{q}{p}}.
\]
We find
\begin{align*}
\vartheta_{j}(v)^{p_{2}}S_{2}(z_{2,j},v)^{p_{2}} &  =S_{2}(z_{2,j}%
,v)^{\frac{p_{2}q}{p}}S_{2}(z_{2,j},v)^{p_{2}}\\
&  =S_{2}(z_{2,j},v)^{q}.
\end{align*}
Thus,
\begin{align*}
\left(  \sum_{j=1}^{m_{2}}S_{2}(z_{2,j},v)^{q}\right)  ^{\frac{1}{p_{2}}} &
=\left(  \sum_{j=1}^{m_{2}}S_{2}(\vartheta_{j}(v)z_{2,j},v)^{p_{2}}\right)
^{\frac{1}{p_{2}}}\\
&  \leq C_{2}\sup_{w\in W_{2}}\left(  \sum_{j=1}^{m_{2}}R_{2}\left(
\vartheta_{j}(v)z_{2,j},w\right)  ^{p_{1}}\right)  ^{\frac{1}{p_{1}}}\\
&  =C_{2}\sup_{w\in W_{2}}\left(  \sum_{j=1}^{m_{2}}\vartheta_{j}(v)^{p_{1}%
}R_{2}\left(  z_{2,j},w\right)  ^{p_{1}}\right)  ^{\frac{1}{p_{1}}}\\
&  \leq C_{2}\sup_{w\in W_{2}}\left[  \left(  \sum_{j=1}^{m_{2}}\vartheta
_{j}(v)^{p_{1}\cdot\frac{p}{p_{1}}}\right)  ^{\frac{p_{1}}{p}}\left(
\sum_{j=1}^{m_{2}}R_{2}\left(  z_{2,j},w\right)  ^{p_{1}\cdot\frac{p_{2}%
}{p_{1}}}\right)  ^{\frac{p_{1}}{p_{2}}}\right]  ^{\frac{1}{p_{1}}}\\
&  =C_{2}\sup_{w\in W_{2}}\left(  \sum_{j=1}^{m_{2}}\vartheta_{j}%
(v)^{p}\right)  ^{\frac{1}{p}}\left(  \sum_{j=1}^{m_{2}}R_{2}\left(
z_{2,j},w\right)  ^{p_{2}}\right)  ^{\frac{1}{p_{2}}}.
\end{align*}
Hence%
\[
\left(  \sum_{j=1}^{m_{2}}S_{2}(z_{2,j},v)^{q}\right)  ^{\frac{1}{p_{2}}%
-\frac{1}{p}}\leq C_{2}\sup_{w\in W_{2}}\left(  \sum_{j=1}^{m_{2}}R_{2}\left(
z_{2,j},w\right)  ^{p_{2}}\right)  ^{\frac{1}{p_{2}}},
\]
i.e.,%
\[
\left(  \sum_{j=1}^{m_{2}}S_{2}(z_{2,j},v)^{q}\right)  ^{\frac{1}{q}}\leq
C\sup_{w\in W_{1}}\left(  \sum_{i=1}^{m_{1}}R_{1}\left(  z_{1,i},w\right)
^{p_{2}}\right)  ^{\frac{1}{p_{2}}}\sup_{w\in W_{2}}\left(  \sum_{j=1}^{m_{2}%
}R_{2}\left(  z_{2,j},w\right)  ^{p_{2}}\right)  ^{\frac{1}{p_{2}}}.
\]
We thus have%
\[
\left(  \sum_{j=1}^{m_{2}}\left(  \sum_{i=1}^{m_{1}}S(z_{1,i},z_{2,j}%
,v)^{p_{2}}\right)  ^{\frac{1}{p_{2}}\cdot q}\right)  ^{\frac{1}{q}}\leq
C\sup_{w\in W_{1}}\left(  \sum_{i=1}^{m_{1}}R_{1}\left(  z_{1,i},w\right)
^{p_{2}}\right)  ^{\frac{1}{p_{2}}}\sup_{w\in W_{2}}\left(  \sum_{j=1}^{m_{2}%
}R_{2}\left(  z_{2,j},w\right)  ^{p_{2}}\right)  ^{\frac{1}{p_{2}}},
\]
and since $p_{2}\leq q$, we have%
\[
\left(  \sum_{i=1}^{m_{1}}\sum_{j=1}^{m_{2}}S(z_{1,i},z_{2,j},v)^{q}\right)
^{\frac{1}{q}}\leq C\sup_{w\in W_{1}}\left(  \sum_{i=1}^{m_{1}}R_{1}\left(
z_{1,i},w\right)  ^{p_{2}}\right)  ^{\frac{1}{p_{2}}}\sup_{w\in W_{2}}\left(
\sum_{j=1}^{m_{2}}R_{2}\left(  z_{2,j},w\right)  ^{p_{2}}\right)  ^{\frac
{1}{p_{2}}},
\]
which finally completes the proof of Theorem \ref{RP}.}
\end{proof}

The reasoning developed in the proof of Theorem \ref{RP} can also be employed
as to produce a more general result. Following the previous set-up, let
$k\geq2$ and $Z_{1}$, $V$ and $W_{1},\cdots,W_{k}$ be arbitrary non-void sets
and $Z_{2},\cdots,Z_{k}$ be vector spaces. For $t=1,...,k,$ let\
\begin{align*}
R_{t}\colon Z_{t}\times W_{t}  &  \longrightarrow\lbrack0,\infty),\text{ and
}\\
S\colon Z_{1}\times\cdots\times Z_{k}\times V  &  \longrightarrow
\lbrack0,\infty)
\end{align*}
be arbitrary mappings satisfying%

\begin{align*}
R_{t}\left(  \lambda z,w\right)   &  =\lambda R_{t}\left(  z,w\right)  ,\\
S\left(  z_{1},...,z_{j-1},\lambda z_{j},z_{j+1},...,z_{k},v\right)   &  =
\lambda S\left(  z_{1},...,z_{j-1},z_{j},z_{j+1},...,z_{k},v\right)
\end{align*}
for all scalars $\lambda\geq0$ and all $j,t=2,\cdots,k.$

\begin{theorem}
[Regularity Principle for $k$-variables]\label{RP3} Let $1\leq p_{1}\leq
p_{2}<\frac{kp_{1}}{k-1},$ and assume
\[
\left(  \sup_{v\in V}\sum_{j_{1}=1}^{m_{1}}\cdots\sum_{j_{k}=1}^{m_{k}%
}S(z_{1,j_{1}},...,z_{k,j_{k}},v)^{p_{1}}\right)  ^{\frac{1}{p_{1}}}\leq C%
{\textstyle\prod\limits_{t=1}^{k}}
\left(  \sup_{w\in W_{t}}\sum_{j=1}^{m_{t}}R_{t}\left(  z_{t,j},w\right)
^{p_{1}}\right)  ^{\frac{1}{p_{1}}},
\]
for all $z_{t,j}\in Z_{t},$ all $t=1,\cdots,k,$ all $j_{t}=1,\cdots,m_{t}$ and
$m_{t}\in\mathbb{N}$. Then
\[
\left(  \sup_{v\in V}\sum_{j_{1}=1}^{m_{1}}\cdots\sum_{j_{k}=1}^{m_{k}%
}S(z_{1,j_{1}},...,z_{k,j_{k}},v)^{\frac{p_{1}p_{2}}{kp_{1}-\left(
k-1\right)  p_{2}}}\right)  ^{\frac{kp_{1}-\left(  k-1\right)  p_{2}}%
{p_{1}p_{2}}}\leq C%
{\textstyle\prod\limits_{t=1}^{k}}
\left(  \sup_{w\in W_{t}}\sum_{j=1}^{m_{t}}R_{t}\left(  z_{t,j},w\right)
^{p_{2}}\right)  ^{\frac{1}{p_{2}}},
\]
for all $z_{t,j}\in Z_{t},$ all $t=1,\cdots,k,$ all $j_{t}=1,\cdots,m_{t}$ and
$m_{t}\in\mathbb{N}$.
\end{theorem}

We omit the details of the proof of Theorem \ref{RP3}. A careful scrutiny of
the second part of the proof of Theorem \ref{RP} and Theorem \ref{RP3} yields
a useful regularity principle itself for anisotropic summability of sequences.
As we shall apply such an estimate in the upcoming sections, we state it as a
separate Theorem.

\begin{theorem}
[Anisotropic Regularity Principle]\label{Ranisot}Let $p_{1},p_{2},r_{1}%
,r_{2}\geq1$ and $p_{3}\geq p_{1}$ and $r_{3}\geq r_{1}$ with%
\[
\frac{1}{r_{1}}-\frac{1}{p_{1}}\leq\frac{1}{r_{3}}-\frac{1}{p_{3}}.
\]
Then%
\begin{align*}
&  \sup_{v\in V}\left(  \sum_{i=1}^{m_{1}}\left(  \sum_{j=1}^{m_{2}}%
S(z_{1,i},z_{2,j},v)^{p_{2}}\right)  ^{\frac{1}{p_{2}}p_{1}}\right)
^{\frac{1}{p_{1}}}\\
&  \leq C\left(  \sup_{w\in W_{1}}\sum_{i=1}^{m_{1}}R_{1}\left(
z_{1,i},w\right)  ^{r_{1}}\right)  ^{\frac{1}{r_{1}}}\left(  \sup_{w\in W_{2}%
}\sum_{j=1}^{m_{2}}R_{2}\left(  z_{2,j},w\right)  ^{r_{2}}\right)  ^{\frac
{1}{r_{2}}},
\end{align*}
for all $z_{1,i},z_{2,j}$ and all $m_{1},m_{2}\in\mathbb{N}$ implies%
\begin{align*}
&  \left(  \sup_{v\in V}\left(  \sum_{i=1}^{m_{1}}\left(  \sum_{j=1}^{m_{2}%
}S(z_{1,i},z_{2,j},v)^{p_{2}}\right)  ^{\frac{1}{p_{2}}\cdot p_{3}}\right)
^{\frac{1}{p_{3}}}\right)  \\
&  \leq C\left(  \sup_{w\in W_{1}}\sum_{i=1}^{m_{1}}R_{1}\left(
z_{1,i},w\right)  ^{r_{3}}\right)  ^{\frac{1}{r_{3}}}\left(  \sup_{w\in W_{2}%
}\sum_{j=1}^{m_{2}}R_{2}\left(  z_{2,j},w\right)  ^{r_{2}}\right)  ^{\frac
{1}{r_{2}}}%
\end{align*}
for all $z_{1,i},z_{2,j}$ and $m_{1},m_{2}\in\mathbb{N}$.
\end{theorem}

For the applications we shall carry on in the next sections, $S$ will be
constant in $V$ and $W_{1},...,W_{k}$ will be compact sets.

\section{New inclusion theorems for multiple summing operators\label{gro}}

It is well known that every absolutely $p$-summing linear operator is
absolutely $q$-summing whenever $1\leq p\leq q$ (see \cite{diestel})$.$ More
generally, absolutely $\left(  p_{1};p_{2}\right)  $-summing operators are
absolutely $\left(  q_{1};q_{2}\right)  $-summing whenever $1\leq p_{j}\leq
q_{j},$ $j=1,2,$ and
\[
\frac{1}{p_{2}}-\frac{1}{p_{1}}\leq\frac{1}{q_{2}}-\frac{1}{q_{1}}.
\]
These kind of results are called inclusion results. For multilinear operators
inclusion results are more challenging. For instance, every multiple
$p$-summing multilinear operator is multiple $q$-summing whenever $1\leq p\leq
q\leq2$, but this is not valid beyond the threshold $2$ (see \cite{p11,pe33}%
)$.$ In this section, as a consequence of the regularity principle, we provide
new inclusion theorems for multiple summing operators.

Henceforth $E,E_{1},\cdots,E_{m},F$ denote Banach spaces over $\mathbb{K}$.
Following classical terminology, the Banach space of all bounded $m$-linear
operators from $E_{1}\times\cdots\times E_{m}$ to $F$ is denoted by
$\mathcal{L}(E_{1},\cdots,E_{m};F)$ and we endow it with the classical sup
norm. The topological dual of $E$ is denoted by $E^{\ast}$ and its closed unit
ball is denoted by $B_{E^{\ast}}.$ Throughout the paper, for $p\in\lbrack1,\infty]$,  $p^\ast$ denotes the conjugate of $p$, that is
$$
	\frac{1}{p} + \frac{1}{p^\ast} = 1.
$$
The convention $1^\ast = \infty$ and $\infty^\ast = 1$ will be adopted. Also, as usual, we consider the Banach spaces of weakly and strongly $p$-summable sequences:
\[
\ell_{p}^{w}(E):=\left\{  (x_{j})_{j=1}^{\infty}\subset E:\left\Vert
(x_{j})_{j=1}^{\infty}\right\Vert _{w,p}:=\sup_{\varphi\in B_{E^{\ast}}%
}\left(  \displaystyle\sum\limits_{j=1}^{\infty}\left\vert \varphi
(x_{j})\right\vert ^{p}\right)  ^{1/p}<\infty\right\}
\]
and
\[
\ell_{p}(E):=\left\{  (x_{j})_{j=1}^{\infty}\subset E:\left\Vert (x_{j}%
)_{j=1}^{\infty}\right\Vert _{p}:=\left(  \displaystyle\sum\limits_{j=1}%
^{\infty}\left\Vert x_{j}\right\Vert ^{p}\right)  ^{1/p}<\infty\right\}  .
\]

For $\mathbf{q}:=(q_{1},\dots,q_{m})\in\lbrack1,\infty)^{m}$, we define the
space of $m$-matrices $\ell_{\mathbf{q}}(E)$ as
\[
\ell_{\mathbf{q}}(E):=\ell_{q_{1}}\left(  \ell_{q_{2}}\left(  \dots\left(
\ell_{q_{m}}(E)\right)  \dots\right)  \right)  .
\]
That is, a vector matrix $\left(  x_{i_{1}\dots i_{m}}\right)  _{i_{1}%
,\dots,i_{m}=1}^{\infty}$ belongs to $\ell_{\mathbf{q}}(E)$ if, and only if,
\[
\left\Vert \left(  x_{i_{1}\dots i_{m}}\right)  _{i_{1},\dots,i_{m}=1}%
^{\infty}\right\Vert _{\ell_{\mathbf{q}}(E)}:=\left(  \sum_{i_{1}=1}^{\infty
}\left(  \dots\left(  \sum_{i_{m}=1}^{\infty}\left\Vert x_{i_{1}\dots i_{m}%
}\right\Vert _{E}^{q_{m}}\right)  ^{\frac{q_{m-1}}{q_{m}}}\dots\right)
^{\frac{q_{1}}{q_{2}}}\right)  ^{\frac{1}{q_{1}}}<\infty.
\]
When $E=\mathbb{K}$, we simply write $\ell_{\mathbf{q}}$. The following
definition will be useful for our purposes.

\begin{definition}
\label{guga030} Let $\mathbf{p}=\left(  p_{1},...,p_{m}\right)  ,\mathbf{q}%
=\left(  q_{1},...,q_{m}\right)  \in\lbrack1,\infty]^{m}$. A multilinear
operator $T\colon E_{1}\times\cdots\times E_{m}\rightarrow F$ is said to be
multiple $\left(  q_{1},...,q_{m};p_{1},...,p_{m}\right)  $-summing if there
exists a constant $C>0$ such that
\begin{equation}
\left(  \sum_{j_{1}=1}^{\infty}\left(  \cdots\left(  \sum_{j_{m}=1}^{\infty
}\left\Vert T\left(  x_{j_{1}}^{(1)},\dots,x_{j_{m}}^{(m)}\right)  \right\Vert
^{q_{m}}\right)  ^{\frac{q_{m-1}}{q_{m}}}\cdots\right)  ^{\frac{q_{1}}{q_{2}}%
}\right)  ^{\frac{1}{q_{1}}}\leq C\prod_{k=1}^{m}\left\Vert \left(  x_{j_{k}%
}^{(k)}\right)  _{j_{k}=1}^{\infty}\right\Vert _{w,p_{k}} \label{01}%
\end{equation}
for all $\left(  x_{j_{k}}^{(k)}\right)  _{j_{k}=1}^{\infty}\in\ell_{p_{k}%
}^{w}\left(  E_{k}\right)  $. We represent the class of all multiple $\left(
q_{1},\dots,q_{m};p_{1},\dots,p_{m}\right)  $--summing operators by
$\Pi_{\left(  q_{1},\dots,q_{m};p_{1},\dots,p_{m}\right)  }^{m}\left(
E_{1},\dots,E_{m};F\right)  $. When $q_{j}=\infty$, the respective sum is
replaced by the sup norm.
\end{definition}

The infimum of all $C>0$ for which (\ref{01}) holds defines a complete norm,
denoted hereafter by $\pi_{(q_{1},\dots,q_{m};p_{1},\dots,p_{m})}(\cdot)$. It
is not hard to verify that
\[
\Pi_{(q_{1},\dots,q_{m};p_{1},\dots,p_{m})}^{m}(E_{1},\dots,E_{m};F)
\]
is a subspace of $\mathcal{L}(E_{1},\dots,E_{m};F)$ and $\Vert\cdot\Vert
\leq\pi_{(q_{1},\dots,q_{m};p_{1},\dots,p_{m})}(\cdot).$ Also, if $q_{j}%
<p_{j}$ for some $j\in\{1,\dots,m\}$, then
\[
\Pi_{\left(  q_{1},\dots,q_{m};p_{1},\dots,p_{m}\right)  }^{m}(E_{1}%
,\dots,E_{m};F)=\{0\}.
\]

The following result associates multiple summing operators and
Hardy--Littlewood inequality.

\begin{theorem}
\label{yh}Let $\left(  p_{1},\dots,p_{m}\right)  \in\lbrack1,\infty]^{m}$. The
following statements are equivalent:

\begin{itemize}
\item[(1)] There is a constant $C>0$ such that for every $T\in\mathcal{L}%
(\ell_{p_{1}},...,\ell_{p_{m}};F)$ the following holds
\[
\left(  \sum_{j_{1}=1}^{\infty}\left(  \cdots\left(  \sum_{j_{m}=1}^{\infty
}\left\Vert T\left(  e_{j_{1}},\dots,e_{j_{m}}\right)  \right\Vert ^{q_{m}%
}\right)  ^{\frac{q_{m-1}}{q_{m}}}\cdots\right)  ^{\frac{q_{1}}{q_{2}}%
}\right)  ^{\frac{1}{q_{1}}}\leq C\Vert T\Vert.
\]

\item[(2)] For all Banach spaces $E_{1},...,E_{m},$ we have
\[
\mathcal{L}(E_{1},...,E_{m};F)=\Pi_{(q_{1},...,q_{m};p_{1}^{\ast}%
,...,p_{m}^{\ast})}^{mult}(E_{1},...,E_{m};F).
\]

\end{itemize}
\end{theorem}

Theorem \ref{yh}, as stated here, is essentially due to D. Pérez-García and I.
Villanueva, see \cite[Corollary 20]{arc}, and its proof rests on the isometric
isomorphisms $\mathcal{L}\left(  \ell_{p^{\ast}},E\right)  \sim\ell_{p}%
^{w}(E)$ and $\mathcal{L}\left(  c_{0},E\right)  \sim\ell_{1}^{w}(E).$ An
advantage of this result for our purposes in subsequent sections is that it
provides a useful way to link Hardy-Littlewood type inequalities to the
language of multiple summing operators; for results on multilinear summing
operators we refer to \cite{bombal,matos} and references therein.

The first application of Theorem \ref{RP3} is an inclusion result for multiple
summing operators which complements, to some extent, the one from \cite{p11}:

\begin{proposition}
\label{878}Let $m\geq2$ be a positive integer and
\[
2\leq r\leq u<\frac{mr}{m-1}.
\]
Then, for any collection of Banach spaces $E_{1},\cdots,E_{m},F$ there holds
\[
\Pi_{(r;r)}^{m}\left(  E_{1},\dots,E_{m};F\right)  \subset\Pi_{(\frac
{ru}{mr-\left(  m-1\right)  u};u)}^{m}\left(  E_{1},\dots,E_{m};F\right)
\]
and the inclusion has norm $1$.
\end{proposition}

\begin{proof}
Using the abstract environment of Section 2, we just need to define
$Z_{j}=E_{j};$ let also $V=\{0\}$ and $T\in\Pi_{(r;r)}^{m}\left(
E_{1},\dots,E_{m};F\right)  $. Now define
\begin{align*}
W_{j} &  =  B_{E_{j}^{\ast}},\\
R_{j}(x,\varphi) &  =\left\vert \varphi(x)\right\vert ,\\
S(x_{1},...,x_{m},v) &  =\left\vert T(x_{1},...,x_{m})\right\vert
\end{align*}
and the proof is a consequence of the Regularity Principle for $k$-variables
(Theorem \ref{RP3}).
\end{proof}

If one carries out the same reasoning employed in the second part of the proof
of the Regularity Principle (Theorem \ref{RP}), the following more general
result can be established:

\begin{proposition}[(Inclusion Theorem)]\label{9990}Let $m$ be a positive integer and $1\leq s\leq
u<\frac{mrs}{mr-s}.$ Then, for any Banach spaces $E_{1},...,E_{m},F$ we have%
\[
\Pi_{(r;s)}^{m}\left(  E_{1},\dots,E_{m};F\right)  \subset\Pi_{(\frac
{rsu}{su+mrs-mru};u)}^{m}\left(  E_{1},\dots,E_{m};F\right)
\]
and the inclusion has norm $1$.
\end{proposition}

Proposition \ref{9990} itself has an interesting application. It provides a
simplified proof of a key technical tool from \cite{teixeira}, that is: for
any positive integer $m$, and any $p>2m,$ there holds%
\begin{equation}
\left(  \sum\limits_{i_{1},\ldots,i_{m}=1}^{\infty}\left\vert U(e_{i_{^{1}}%
},\ldots,e_{i_{m}})\right\vert ^{\frac{2p}{p-2m}}\right)  ^{\frac{p-2m}{2p}%
}\leq\left\Vert U\right\Vert , \label{654}%
\end{equation}
for all $m$--linear forms $U\colon\ell_{p}^{n}\times\cdots\times\ell_{p}%
^{n}\rightarrow\mathbb{K}$ and all positive integers $n$.

Indeed, as every continuous $m$-linear form $T$ is multiple $\left(
2;1\right)  $-summing with constant $1$, one simply takes $\left(
r,s,u\right)  =\left(  2,1,p^{\ast}\right)  $ in the statement of Proposition
\ref{9990} and arrives at (\ref{654}).

As a matter of fact, every continuous $m$-linear form $T$ is actually
multiple
\[
\left(  2,...,2;1,...,1,2\right)  \text{--summing}%
\]
with constant $1$. If one uses this stronger information, one can actually
improves (\ref{654}) as it yields the $\ell_{\frac{2p}{p-2m+2}}$-norm of
$|U(e_{i_{^{1}}},\ldots,e_{i_{m}})|$ is controlled by $\Vert U\Vert$.

\section{Grothendieck-type theorems}

Every continuous linear operator from $\ell_{1}$ into $\ell_{2}$ is absolutely
$\left(  q,p\right)  $--summing for every $q\geq p\geq1;$ this result is a
trademark theorem proven by Grothendieck in his seminal 1950's
\textquotedblleft R\'esum\'e", \cite{GG} --- for recent monographs on
Grothendieck's R\'esum\'e we refer to \cite{RR,PISIER}. More precisely, the result
asserts that
\[
\left(
{\textstyle\sum\limits_{j=1}^{m}}
\left\Vert u(x_{j})\right\Vert ^{q}\right)  ^{\frac{1}{q}}\leq C\sup
_{\varphi\in B_{\ell_{\infty}}}\left(
{\textstyle\sum\limits_{j=1}^{m}}
\left\vert \varphi(x_{j})\right\vert ^{p}\right)  ^{\frac{1}{p}}%
\]
for all continuous linear operators $u\colon\ell_{1}\rightarrow\ell_{2}.$ This
result is in fact very special as illustrated by the following result due to
Lindenstrauss and Pelczynski (\cite{lindenstrauss}): if $E,F$ are infinite
dimensional Banach spaces and $E$ has unconditional Schauder basis, and every
continuous linear operator from $E$ to $F$ is absolutely $1$-summing, then
$E=\ell_{1}$ and $F$ is a Hilbert space. In the multilinear setting, every
continuous $m$-linear operator from $\ell_{1}\times\cdots\times\ell_{1}$ into
$\ell_{2}$ is multiple $\left(  q,p\right)  $--summing for every $1\leq
p\leq2$ and every $q\geq p$ (\cite[Theorems 5.1 and 5.2]{bombal} and
\cite{p11})$,$ and, when $m\geq2$, this result is no longer valid if $p>2$
(\cite[Theorem 3.6]{pe33}). The results of the previous section provide estimates for values
of $q$ for which every continuous $m$-linear operator $T\colon\ell_{1}%
\times\cdots\times\ell_{1}\rightarrow\ell_{2}$ is multiple $\left(
q,p\right)  $-summing when $p=2+\epsilon$, for $\epsilon$ small. However,
since $\left(  \ell_{1},\ell_{2}\right)  $ is a quite special pair of Banach
spaces for summability purposes, we are able to provide a definitive result
with all pairs of $\left(  q,p\right)  $ for which $\Pi_{(q;p)}^{m}\left(
^{m}\ell_{1};\ell_{2}\right)  =\mathcal{L}\left(  ^{m}\ell_{1};\ell
_{2}\right)  .$

\begin{theorem}
Let $m\ge 2$ be a positive integer and $1\leq p\leq q<\infty.$ Then%
\[
\Pi_{(q;p)}^{m}\left(  ^{m}\ell_{1};\ell_{2}\right)  =\mathcal{L}\left(
^{m}\ell_{1};\ell_{2}\right)
\]
if and only if $p\leq2$ or $q>p>2.$
\end{theorem}

\begin{proof} If $p\leq2,$ by \cite{p11} we know that every continuous $m$-linear
operator is multiple $\left(  q;p\right)  $-summing for all $q\geq p.$ If
$p>2$, by \cite{pe33} we know that $\Pi_{(p;p)}^{m}\left(  ^{m}\ell_{1}%
;\ell_{2}\right)  \neq\mathcal{L}\left(  ^{m}\ell_{1};\ell_{2}\right)  .$ It
remains to prove that $\Pi_{(q;p)}^{m}\left(  ^{m}\ell_{1};\ell_{2}\right)
=\mathcal{L}\left(  ^{m}\ell_{1};\ell_{2}\right)  $ for all $q>p>2.$ So, let
us consider $q>p>2.$ By \cite[Proposition 3.6]{bot} we know that
\begin{equation}
\Pi_{(q;p)}^{m}\left(  ^{m}\ell_{1};\mathbb{K}\right)  =\mathcal{L}\left(
^{m}\ell_{1};\mathbb{K}\right)  . \label{88k}%
\end{equation}
It is not difficult to prove that from (\ref{88k}) we conclude that every
continuous $m$-linear operator $T:\ell_{1}\times\cdots\times\ell
_{1}\rightarrow F$ sends weakly $p$-summable sequences into weakly
$q$-summable sequences, regardless of the Banach space $F$. More precisely,%
\[
\left(  T\left(  x_{j_{1}}^{(1)},...,x_{j_{m}}^{(m)}\right)  \right)
_{j_{1},...,j_{m}=1}^{\infty}\in\ell_{q}^{w}(F)
\]
whenever%
\[
\left(  x_{j_{k}}^{(k)}\right)  _{j_{k}=1}^{\infty}\in\ell_{p}^{w}%
(E_{k}),\text{ }k=1,...,m.
\]
Now, considering $\Psi:\ell_{1}\times\cdots\times\ell_{1}\rightarrow\ell
_{1}\widehat{\otimes}_{\pi}\cdots\widehat{\otimes}_{\pi}\ell_{1}$ given by%
\[
\Psi\left(  x^{(1)},...,x^{(m)}\right)  =x^{(1)}\widehat{\otimes}_{\pi}%
\cdots\widehat{\otimes}_{\pi}x^{(m)},
\]
we conclude that
\[
\left(  x_{j_{1}}^{(1)}\widehat{\otimes}_{\pi}\cdots\widehat{\otimes}_{\pi
}x_{j_{m}}^{(m)}\right)  _{j_{1},...,j_{m}=1}^{\infty}\in\ell_{q}^{w}(\ell
_{1}\widehat{\otimes}_{\pi}\cdots\widehat{\otimes}_{\pi}\ell_{1})
\]
whenever
\[
\left(  x_{j_{k}}^{(k)}\right)  _{j_{k}=1}^{\infty}\in\ell_{p}^{w}(\ell
_{1}),\text{ }k=1,...,m.
\]
Let $\widetilde{T}:\ell_{1}\widehat{\otimes}_{\pi}\cdots\widehat{\otimes}%
_{\pi}\ell_{1}\rightarrow\ell_{2}$ be the linearization of $T.$ Since
$\ell_{1}\widehat{\otimes}_{\pi}\cdots\widehat{\otimes}_{\pi}\ell_{1}$ is
isometrically isomorphic to $\ell_{1}$, then $\widetilde{T}$ is absolutely
$q$-summing and, for $\left(  x_{j_{k}}^{(k)}\right)  _{j_{k}=1}^{\infty}%
\in\ell_{p}^{w}(\ell_{1}),$ $k=1,...,m,$ we have%
\begin{align*}
\left(
{\textstyle\sum\limits_{j_{1},...,j_{m}=1}^{\infty}}
\left\Vert T\left(  x_{j_{1}}^{(1)},...,x_{j_{m}}^{(m)}\right)  \right\Vert
^{q}\right)  ^{\frac{1}{q}}  &  =\left(
{\textstyle\sum\limits_{j_{1},...,j_{m}=1}^{\infty}}
\left\Vert \widetilde{T}\left(  x_{j_{1}}^{(1)}\widehat{\otimes}_{\pi}%
\cdots\widehat{\otimes}_{\pi}x_{j_{m}}^{(m)}\right)  \right\Vert ^{q}\right)
^{\frac{1}{q}}\\
&  \leq C\left\Vert \left(  x_{j_{1}}^{(1)}\widehat{\otimes}_{\pi}%
\cdots\widehat{\otimes}_{\pi}x_{j_{m}}^{(m)}\right)  _{j_{1},...,j_{m}%
=1}^{\infty}\right\Vert _{w,q}<\infty
\end{align*}
and the proof is done.
\end{proof}

As a matter of fact a similar result holds in a more general setting:

\begin{theorem}
Let $m\geq2$ be an integer and $F$ be a Banach space. If $\Pi_{(q;p)}\left(
\ell_{1};F\right)  =\mathcal{L}\left(  \ell_{1};F\right)  $, then%
\[
\Pi_{(q+\delta;p)}^{m}\left(  ^{m}\ell_{1};F\right)  =\mathcal{L}\left(
^{m}\ell_{1};F\right)
\]
for all $\delta>0.$
\end{theorem}

\begin{proof} By \cite[Proposition 3.6]{bot} we know that%
\[
\Pi_{(p+\varepsilon;p)}^{m}\left(  ^{m}\ell_{1};\mathbb{K}\right)
=\mathcal{L}\left(  ^{m}\ell_{1};\mathbb{K}\right)
\]
for all $\varepsilon>0.$ As in the previous proof we know that
\[
\left(  x_{j_{1}}^{(1)}\widehat{\otimes}_{\pi}\cdots\widehat{\otimes}_{\pi
}x_{j_{m}}^{(m)}\right)  _{j_{1},...,j_{m}=1}^{\infty}\in\ell_{p+\varepsilon
}^{w}(\ell_{1}\widehat{\otimes}_{\pi}\cdots\widehat{\otimes}_{\pi}\ell_{1})
\]
whenever
\[
\left(  x_{j_{k}}^{(k)}\right)  _{j_{k}=1}^{\infty}\in\ell_{p}^{w}(\ell
_{1}),\text{ }k=1,...,m.
\]
Let $\widetilde{T}:\ell_{1}\widehat{\otimes}_{\pi}\cdots\widehat{\otimes}%
_{\pi}\ell_{1}\rightarrow F$ be the linearization of $T.$ Since $\ell
_{1}\widehat{\otimes}_{\pi}\cdots\widehat{\otimes}_{\pi}\ell_{1}$ is
isometrically isomorphic to $\ell_{1}$, then $\widetilde{T}$ is absolutely
$\left(  q;p\right)  $-summing and hence for any $\delta>0$ there is a
$\varepsilon>0$ such that $\widetilde{T}$ is $\left(  q+\delta;p+\varepsilon
\right)  $-summing. Therefore, for $\left(  x_{j_{k}}^{(k)}\right)  _{j_{k}%
=1}^{\infty}\in\ell_{p}^{w}(\ell_{1}),$ $k=1,...,m,$ we have%
\begin{align*}
\left(
{\textstyle\sum\limits_{j_{1},...,j_{m}=1}^{\infty}}
\left\Vert T\left(  x_{j_{1}}^{(1)},...,x_{j_{m}}^{(m)}\right)  \right\Vert
^{q+\delta}\right)  ^{\frac{1}{q+\delta}}  &  =\left(
{\textstyle\sum\limits_{j_{1},...,j_{m}=1}^{\infty}}
\left\Vert \widetilde{T}\left(  x_{j_{1}}^{(1)}\widehat{\otimes}_{\pi}%
\cdots\widehat{\otimes}_{\pi}x_{j_{m}}^{(m)}\right)  \right\Vert ^{q+\delta
}\right)  ^{\frac{1}{q+\delta}}\\
&  \leq C\left\Vert \left(  x_{j_{1}}^{(1)}\widehat{\otimes}_{\pi}%
\cdots\widehat{\otimes}_{\pi}x_{j_{m}}^{(m)}\right)  _{j_{1},...,j_{m}%
=1}^{\infty}\right\Vert _{w,p+\varepsilon}<\infty
\end{align*}
and the proof is complete.
\end{proof}

\section{Sharp anisotropic Hardy--Littlewood inequality \label{hl}}

The investigation of bilinear forms acting on sequence spaces goes back to the
pioneering work of Hilbert on his famous double-series theorem and, throughout
the 20th century, has attracted the attention of leading mathematicians as
Weyl, Toeplitz, Schur, Nehari, see \cite{nehari}, \cite{schur}, \cite{toe} and
references therein. A trademark of the field comes from Littlewood's solution,
\cite{Litt}, to the problem posed by P.J. Daniell; the famous Littlewood's
$4/3$ inequality is an estimate that represents the extremal case $p=q=\infty$
in \eqref{inicio}. Bohnenblust and Hille, \cite{bh}, obtained important
generalizations of Littlewood's $4/3$ inequality to the setting of $m$--linear
operators and few years later, Hardy and Littlewood proved a series of
inequalities for bilinear forms acting on $\ell_{p}\times\ell_{q}$ spaces,
with $\frac{1}{p}+\frac{1}{q}<1$, which would launch a new and promising line
of investigation; named thereafter Hardy-Littlewood type inequalities. The key
objective of study is to control
\begin{equation}
\left(  \sum_{i=1}^{n}\left(  \sum_{j=1}^{n}\left\vert T(e_{i},e_{j}%
)\right\vert ^{a}\right)  ^{\frac{1}{a}\cdot b}\right)  ^{\frac{1}{b}}
\label{quant}%
\end{equation}
for all norm-1 bilinear operators $T\colon\ell_{p}^{n}\times\ell_{q}%
^{n}\rightarrow\mathbb{K}$ and all positive integers $n$.

The search of optimal ranges of exponents for universal summability of
\eqref{quant} has been carried out by direct and indirect approaches
permeating the theory and some sectional answers have been collected
throughout the decades. Partial solutions can be found, for instance, in
\cite{alb} and \cite{tonge}. While these represented important advances in the
global understanding of the problem, the results proven thus far are limited
when {it} comes to determining the whole spectrum of admissible exponents.
This is, indeed, a subtle issue which resembles the problem of Schur
multipliers investigated by Bennett in \cite{bennett}.

Before we state the main Theorem of this section, we highlight that this is
more than just a beautiful mathematical puzzle. Even when restricted to
\textit{classical} isotropic multiple summing, enlarging the studies to the
anisotropic setting reveals a number of important nuances that could not be
perceived otherwise. This is a somewhat common procedure in the realm of
mathematics --- solving real problems through complex methods being probably
the most emblematic example.

As an application of the regularity principle, we will prove the following
complete characterization of all admissible anisotropic exponents for the
Hardy-Littlewood inequality:

\begin{theorem}
\label{main} Let $p,q \in\lbrack2,\infty]$ with $\frac{1}{p}+\frac{1}{q}<1$,
and $a,b>0.$ The following assertions are equivalent:

\begin{itemize}
\item[(a)] There is a constant $C_{p,q,a,b}\geq1$ such that
\begin{equation}
\left(  \sum_{i=1}^{n}\left(  \sum_{j=1}^{n}\left\vert U(e_{i},e_{j}%
)\right\vert ^{a}\right)  ^{\frac{1}{a} \cdot b}\right)  ^{\frac{1}{b}}\leq
C_{p,q,a,b}\left\Vert U\right\Vert , \label{811}%
\end{equation}
for all bilinear operators $U \colon\ell_{p}^{n}\times\ell_{q}^{n}%
\rightarrow\mathbb{K}$ and all positive integers $n$.

\smallskip

\item[(b)] The exponents $a,b$ satisfy $\left(  a,b\right)  \in\lbrack\frac
{q}{q-1},\infty)\times\lbrack\frac{pq}{pq-p-q},\infty)$ and
\begin{equation}
\frac{1}{a}+\frac{1}{b}\leq\frac{3}{2}-\left(  \frac{1}{p}+\frac{1}{q}\right)
. \label{obey}%
\end{equation}

\end{itemize}
\end{theorem}

The proof of Theorem \ref{main} will be developed in the sequel. As commented
above, the main technical novelty of the proof is the regularity principle
established in Section \ref{Sct RP}, which, in this particular case, reveals
sharp and subtle inclusion properties that were not accessible by preceding methods.

\medskip

We start off by recalling the following inequality sometimes credited to
Minkowski (see, for instance, \cite{garling}): if $1\leq p\leq q,$ then
\begin{equation}
\left(  \sum\limits_{i=1}^{\infty}\left(  \sum\limits_{j=1}^{\infty}\left\vert
a_{ij}\right\vert ^{p}\right)  ^{\frac{1}{p}\cdot q}\right)  ^{\frac{1}{q}%
}\leq\left(  \sum\limits_{j=1}^{\infty}\left(  \sum\limits_{i=1}^{\infty
}\left\vert a_{ij}\right\vert ^{q}\right)  ^{\frac{1}{q}\cdot p}\right)
^{\frac{1}{p}} \label{7j}%
\end{equation}
for all sequence of scalars matrices $\left(  a_{ij}\right)  .$ We will also
make use of the main result from \cite{alb}, which we state here for the
readers' convenience. From now on
\[
\mathbf{p}:=(p_{1},\ldots,p_{m})\in\lbrack1,\infty]^{m}
\]
and we denote
\[
\left\vert \frac{1}{\mathbf{p}}\right\vert :=\frac{1}{p_{1}}+\cdots+\frac
{1}{p_{m}}.
\]

\begin{theorem}
[Generalized Hardy--Littlewood inequality \cite{alb}]\label{BH_HL} Let
$\mathbf{p}:=(p_{1},\dots,p_{m})\in\lbrack1,\infty]^{m}$ be such that
\[
\left\vert \frac{1}{\mathbf{p}}\right\vert \leq\frac{1}{2} \quad\text{and}
\quad\mathbf{q}:=(q_{1},\dots,q_{m})\in\left[  \left(  1-\left\vert \frac
{1}{\mathbf{p}}\right\vert \right)  ^{-1},2\right]  ^{m}.
\]
The following are equivalent:

\begin{itemize}
\item[(1)] There is a constant $C_{m,\mathbf{p},\mathbf{q}}^{\mathbb{K}}\geq1$
such that
\begin{equation}
\left(  \sum\limits_{j_{1}=1}^{\infty}\left(  \cdots\left(  \sum
\limits_{j_{m}=1}^{\infty}\left\vert A\left(  e_{j_{1}},\ldots,e_{j_{m}%
}\right)  \right\vert ^{q_{m}}\right)  ^{\frac{q_{m-1}}{q_{m}}}\cdots\right)
^{\frac{q_{1}}{q_{2}}}\right)  ^{\frac{1}{q_{1}}}\leq C_{m,\mathbf{p}%
,\mathbf{q}}^{\mathbb{K}}\left\Vert A\right\Vert \label{i89}%
\end{equation}
for all $m$-linear forms $A \colon\ell_{p_{1}}^{n}\times\cdots\times
\ell_{p_{m}}^{n}\rightarrow\mathbb{K}$ and all positive integers $n$.

\item[(2)] The inequality
\[
\frac{1}{q_{1}}+\cdots+\frac{1}{q_{m}}\leq\frac{m+1}{2}-\left\vert {\frac
{1}{\mathbf{p}}}\right\vert
\]
is verified.
\end{itemize}
\end{theorem}

We begin with a Lemma which plays a key role in the solution of the classification problem for all sharp exponents in the anisotropic Hardy--Littlewood inequality. In the heart of its proof lies the Regularity Principle established in Section \ref{Sct RP}.

\begin{lemma}
\label{prevlemma} \label{7m}Let $E_{1},E_{2}$ be Banach spaces, $p\in
(2,\infty)$, and $q\in\lbrack2,\infty]$. Then every continuous $2$-linear
operator $U\colon E_{1}\times E_{2}\rightarrow\mathbb{K}$ is multiple $\left(
\frac{2p}{p-2},\frac{q}{q-1};p^{\ast},q^{\ast}\right)  $-summing.
\end{lemma}

\begin{proof}
Initially we observe that from Theorem \ref{BH_HL} and Theorem \ref{yh} there is a constant $C_{0}$ such that%
\[
\left(  \sum_{j=1}^{\infty}\left(  \sum_{i=1}^{\infty}\left\vert U(x_{i}%
,y_{j})\right\vert ^{2}\right)  ^{\frac{1}{2} \cdot 1}\right)  ^{\frac{1}{1}%
}\leq C_{0}\left\Vert U\right\Vert \left\Vert \left(  x_{i}\right)
\right\Vert _{w,1}\left\Vert \left(  y_{j}\right)  \right\Vert _{w,1}.
\]
By the Anisotropic Regularity Principle we have%
\[
\left(  \sum_{j=1}^{\infty}\left(  \sum_{i=1}^{\infty}\left\vert U(x_{i}%
,y_{j})\right\vert ^{2}\right)  ^{\frac{1}{2}\cdot q^{\ast}}\right)
^{\frac{1}{q^{\ast}}}\leq C_{0}\left\Vert U\right\Vert \left\Vert \left(
x_{i}\right)  \right\Vert _{w,1}\left\Vert \left(  y_{j}\right)  \right\Vert
_{w,q^{\ast}}.
\]
Since $q^{\ast}\leq2,$ by the Minkowski inequality (\ref{7j}) there holds%
\[
\left(  \sum_{i=1}^{\infty}\left(  \sum_{j=1}^{\infty}\left\vert U(x_{i}%
,y_{j})\right\vert ^{q^{\ast}}\right)  ^{\frac{1}{q^{\ast}} \cdot 2}\right)
^{\frac{1}{2}}\leq C_{0}\left\Vert U\right\Vert \left\Vert \left(
x_{i}\right)  \right\Vert _{w,1}\left\Vert \left(  y_{j}\right)  \right\Vert
_{w,q^{\ast}}.
\]
Finally, by the Anisotropic Regularity Principle we obtain%
\[
\left(  \sum_{i=1}^{\infty}\left(  \sum_{j=1}^{\infty}\left\vert U(x_{i}%
,y_{j})\right\vert ^{q^{\ast}}\right)  ^{\frac{1}{q^{\ast}} \cdot\frac
{2p}{p-2}}\right)  ^{\frac{p-2}{2p}}\leq C_{0}\left\Vert U\right\Vert
\left\Vert \left(  x_{i}\right)  \right\Vert _{w,p^{\ast}}\left\Vert \left(
y_{j}\right)  \right\Vert _{w,q^{\ast}},
\]
and the Lemma is proven.
\end{proof}

Next, we will make use of a H\"older-type inequality essentially due to
Benedek and Panzone \cite{ben} and a generalization of the
Kahane--Salem--Zygmund inequality; to assist the readers, we state them both
as we will need and cite \cite[page 50]{ff} and \cite{alb}, \cite{boas} for
their proofs.

\begin{theorem}
[Interpolative H\"older inequality]\label{gen.interp} Let $n$ be a positive
integer and
\[
q_{1}, \, q_{2}, \, q_{1}(k), \, q_{2}(k)\in\lbrack1,\infty]
\]
with $k=1,2$ be such that
\[
\left(  \frac{1}{q_{1}},\frac{1}{q_{2}}\right)  =\theta\left(  \frac{1}%
{q_{1}(1)},\frac{1}{q_{2}(1)}\right)  +\left(  1-\theta\right)  \left(
\frac{1}{q_{1}(2)},\frac{1}{q_{2}(2)}\right)
\]
for a certain $\theta\in\lbrack0,1].$ Then
\begin{align*}
&  \left(  \sum_{i_{1}=1}^{n}\left(  \sum_{i_{2}=1}^{n}|a_{i_{1},i_{2}%
}|^{q_{2}}\right)  ^{\frac{q_{1}}{q_{2}}}\right)  ^{\frac{1}{q_{1}}}\\
&  {\small \leq}\left[  \left(  \sum_{i_{1}=1}^{n}\left(  \sum_{i_{2}=1}%
^{n}|a_{i_{1},i_{2}}|^{q_{2}(1)}\right)  ^{\frac{q_{1}(1)}{q_{2}(1)}}\right)
^{\frac{1}{q_{1}(1)}}\right]  ^{\theta} \cdot\left[  \left(  \sum_{i_{1}%
=1}^{n}\left(  \sum_{i_{2}=1}^{n}|a_{i_{1},i_{2}}|^{q_{2}(2)}\right)
^{\frac{q_{1}(2)}{q_{2}(2)}}\right)  ^{\frac{1}{q_{1}(2)}}\right]  ^{1-\theta
},
\end{align*}
for all positive integers $n.$
\end{theorem}

\begin{theorem}
[Kahane--Salem--Zygmund inequality]\label{sss} Let $m,n\geq1$ and
$p_{1},\cdots, p_{m}\in\left[  2,\infty\right]  .$ There is an universal
constant $C_{m}$, depending only on $m$, and an $m$-linear mapping $A_{n}
\colon\ell_{p_{1}}^{n}\times\cdots\times\ell_{p_{m}}^{n}\rightarrow\mathbb{K}$
of the form
\[
A_{n}(z^{(1)},...,z^{(m)})=\displaystyle\sum_{i_{1},...,i_{m}=1}^{n}\pm
z_{i_{1}}^{(1)}\cdots z_{i_{m}}^{(m)}%
\]
such that
\[
\Vert A_{n}\Vert\leq C_{m}n^{\frac{m+1}{2}-\left(  \frac{1}{p_{1}}%
+\cdots+\frac{1}{p_{m}}\right)  }.
\]

\end{theorem}

We have gathered all the tools needed to deliver a proof of the main result of
this section classifying all possible exponents $a,b>0$ for which there is a
Hardy--Littlewood-type inequality for bilinear forms $U \colon\ell_{p}%
^{n}\times\ell_{q}^{n}\rightarrow\mathbb{K}$ with $p,q\in\lbrack2,\infty]$ and
$\frac{1}{p}+\frac{1}{q}<1:$

\bigskip

\noindent\textit{Proof of Theorem \ref{main}.}

\medskip We will divide our analysis in two cases: when $p>2$ and when $p=2$.
Let us start the proof in the case $p>2$.

\noindent(b)$\Rightarrow$(a). Suppose that $\left(  a,b\right)  \in
\lbrack\frac{q}{q-1},2]\times\lbrack\frac{pq}{pq-p-q},\frac{2p}{p-2}]$ with
\begin{equation}
\label{9hb}\frac{1}{a}+\frac{1}{b}\leq\frac{3}{2}-\left(  \frac{1}{p}+\frac
{1}{q}\right)  .
\end{equation}

It suffices to consider the case in which we have an equality in (\ref{9hb}).
There is a $\theta\in\lbrack0,1]$ such that
\[
\left(  \frac{1}{b},\frac{1}{a}\right)  =\theta\left(  \frac{p-2}{2p}%
,\frac{q-1}{q}\right)  +\left(  1-\theta\right)  \left(  \frac{1}{\lambda
},\frac{1}{2}\right)  ,
\]
where $\lambda=\frac{pq}{pq-p-q}.$ Applying Theorem \ref{gen.interp}, Lemma
\ref{prevlemma}, combined with \cite[Theorems 1 and 2]{hardy}, we conclude
that there are constants $C_{0},C_{1}\geq1$ such that
\begin{align*}
&  \left(  \sum_{i=1}^{n}\left(  \sum_{j=1}^{n}\left\vert U(e_{i}%
,e_{j})\right\vert ^{a}\right)  ^{\frac{1}{a} \cdot b}\right)  ^{\frac{1}{b}%
}\\
&  \leq\left[  \left(  \sum_{i=1}^{n}\left(  \sum_{j=1}^{n}\left\vert
U(e_{i},e_{j})\right\vert ^{q^{\ast}}\right)  ^{\frac{1}{q^{\ast}} \cdot
\frac{2p}{p-2}}\right)  ^{\frac{p-2}{2p}}\right]  ^{\theta} \cdot\left[
\left(  \sum_{i=1}^{n}\left(  \sum_{j=1}^{n}\left\vert U(e_{i},e_{j}%
)\right\vert ^{2}\right)  ^{\frac{\lambda}{2}}\right)  ^{\frac{1}{\lambda}%
}\right]  ^{1-\theta}\\
&  \leq\left(  C_{0}\left\Vert U\right\Vert \right)  ^{\theta} \cdot\left(
C_{1}\left\Vert U\right\Vert \right)  ^{1-\theta}\\
&  =\left(  C_{0}\right)  ^{\theta}C_{1}^{1-\theta}\left\Vert U\right\Vert .
\end{align*}

The case $\left(  a,b\right)  \in\left(  \lbrack\frac{q}{q-1},2]\times
\lbrack\frac{2p}{p-2},\infty)\right)  \cup\left(  \lbrack2,\infty
)\times\lbrack\frac{pq}{pq-p-q},\infty)\right)  $ is a straightforward
consequence of the previous result for $\theta=0$ and $\theta=1$.

\medskip

\noindent(a)$\Rightarrow$(b). For any positive integer $n$, consider the
bilinear operator%
\[
U_{n} \colon\ell_{p}^{n}\times\ell_{q}^{n}\rightarrow\mathbb{K}%
\]
given by%
\[
U_{n}\left(  x,y\right)  =x_{1}\sum\limits_{j=1}^{n}y_{j}.
\]
Since $\left\Vert U_{n}\right\Vert =n^{\frac{1}{q^{\ast}}}$, plugging $U_{n}$
into (\ref{811}) we conclude that%
\[
n^{\frac{1}{a}}\leq C_{p,q}n^{\frac{1}{q^{\ast}}}%
\]
for a certain constant $C_{p,q}$, and since $n$ is arbitrary we conclude that
$a\geq\frac{q}{q-1}$. Now we consider the bilinear operator%
\[
V_{n}\colon\ell_{p}^{n}\times\ell_{q}^{n}\rightarrow\mathbb{K}%
\]
given by%
\[
V_{n}\left(  x,y\right)  =\sum\limits_{j=1}^{n}x_{j}y_{j}.
\]
Since $\left\Vert V_{n}\right\Vert =n^{1-\left(  \frac{1}{p}+\frac{1}%
{q}\right)  }$, plugging $V_{n}$ into (\ref{811}) we conclude that%
\[
n^{\frac{1}{b}}\leq C_{p,q}n^{1-\left(  \frac{1}{p}+\frac{1}{q}\right)  }%
\]
for a certain constant $C_{p,q}$, and thus%
\[
b\geq\frac{pq}{pq-p-q}.
\]
It remains to verify that for $\left(  a,b\right)  \in\lbrack\frac{q}%
{q-1},\infty)\times\lbrack\frac{pq}{pq-p-q},\infty)$ the exponents must obey
(\ref{obey}). Let $A_{n} \colon\ell_{p}^{n}\times\ell_{q}^{n}\rightarrow
\mathbb{K}$ be the bilinear form given by the Kahane--Salem--Zygmund
inequality. Using (\ref{811}) with $A_{n}$ we obtain%
\[
n^{\frac{1}{a}+\frac{1}{b}}\leq n^{\frac{3}{2}-\left(  \frac{1}{p}+\frac{1}%
{q}\right)  },
\]
and thus%
\[
\frac{1}{a}+\frac{1}{b}\leq\frac{3}{2}-\left(  \frac{1}{p}+\frac{1}{q}\right)
\]

This concludes the proof of Theorem \ref{main} when $p>2$.

\medskip

Let us turn our attention to the case $p=2$. Initially, we revisit the proof
of Lemma \ref{7m} and note that if $E_{1}, \ E_{2}$ are Banach spaces and
$q\in\lbrack2,\infty),$ then every continuous $2$-linear operator $T \colon
E_{1}\times E_{2}\rightarrow\mathbb{K}$ is multiple $\left(  \infty,q^{\ast
};2,q^{\ast}\right)  $-summing. Thus, there is a constant $C_{0}\geq1$ such
that
\begin{equation}
\left(  \sup_{i}\left(  \sum_{j=1}^{n}\left\vert U(e_{i},e_{j})\right\vert
^{q^{\ast}}\right)  ^{\frac{1}{q^{\ast}}}\right)  \leq C_{0}\left\Vert
U\right\Vert , \label{mmmm8}%
\end{equation}
for all $U \colon\ell_{2}^{n}\times\ell_{q}^{n}\rightarrow\mathbb{K}$ and all
positive integers $n$. We proceed with the proof. \medskip

\noindent(b)$\Rightarrow$(a). Suppose that $\left(  a,b\right)  \in(\frac
{q}{q-1},2]\times\lbrack\frac{2q}{q-2},\infty)$ with
\begin{equation}
\label{poy1}\frac{1}{a}+\frac{1}{b}\leq\frac{1}{q^{\ast}}.
\end{equation}

It suffices to consider the case in which we have an equality in (\ref{poy1}).
We can find $\theta\in\lbrack0,1)$ such that
\[
\left(  \frac{1}{b},\frac{1}{a}\right)  =\theta\left(  \frac{1}{\infty}%
,\frac{q-1}{q}\right)  +\left(  1-\theta\right)  \left(  \frac{q-2}{2q}%
,\frac{1}{2}\right)  .\text{ }%
\]
By Theorem \ref{gen.interp}, \cite[Theorem 2]{hardy} and (\ref{mmmm8}) there exist
constants $C_{0},C_{1}$ such that
\begin{align*}
&  \left(  \sum_{i=1}^{n}\left(  \sum_{j=1}^{n}\left\vert U(e_{i}%
,e_{j})\right\vert ^{a}\right)  ^{\frac{1}{a}\cdot b}\right)  ^{\frac{1}{b}}\\
&  \leq\left[  \left(  \sup_{i}\left(  \sum_{j=1}^{n}\left\vert U(e_{i}%
,e_{j})\right\vert ^{q^{\ast}}\right)  ^{\frac{1}{q^{\ast}}}\right)  \right]
^{\theta}\cdot\left[  \left(  \sum_{i=1}^{n}\left(  \sum_{j=1}^{n}\left\vert
U(e_{i},e_{j})\right\vert ^{2}\right)  ^{\frac{1}{2} \cdot\frac{2q}{q-2}%
}\right)  ^{\frac{q-2}{2q}}\right]  ^{1-\theta}\\
&  \leq\left(  C_{0}\left\Vert U\right\Vert \right)  ^{\theta} \cdot\left(
C_{1}\left\Vert U\right\Vert \right)  ^{1-\theta}\\
&  =\left(  C_{0}\right)  ^{\theta}C_{1}^{1-\theta}\left\Vert U\right\Vert .
\end{align*}
The case $\left(  a,b\right)  \in\lbrack2,\infty)\times\lbrack\frac{2q}%
{q-2},\infty)$ is a straightforward consequence of the previous result for
$\theta=0$. The proof of (a)$\Rightarrow$(b) is a consequence of the
Kahane--Salem--Zygmund inequality as before. This concludes the proof of
Theorem \ref{main}. \hfill$\square$

\bigskip

We conclude this section {commenting on the case where the sums are in the
reverse order. Arguing by symmetry, the following also holds}: let
$p\in\lbrack2,\infty]$, $q\in\lbrack2,\infty]$ with $\frac{1}{p}+\frac{1}%
{q}<1,$ and $a,b>0.$ The following assertions are equivalent:

\begin{itemize}
\item[(a)] There is a constant $C\geq1$ such that
\[
\left(  \sum_{j=1}^{n}\left(  \sum_{i=1}^{n}\left\vert U(e_{i},e_{j}
)\right\vert ^{a}\right)  ^{\frac{1}{a}\cdot b}\right)  ^{\frac{1}{b}}\leq
C\left\Vert U\right\Vert
\]
for all bilinear operators $U \colon\ell_{p}^{n}\times\ell_{q}^{n}%
\rightarrow\mathbb{K}$ and all positive integers $n$.

\item[(b)] The exponents $a,b$ satisfy $\left(  a,b\right)  \in\lbrack\frac{p}
{p-1},\infty)\times\lbrack\frac{pq}{pq-p-q},\infty)$ with
\[
\frac{1}{a}+\frac{1}{b}\leq\frac{3}{2}-\left(  \frac{1}{p}+\frac{1}{q}\right)
.
\]

\end{itemize}

\bigskip

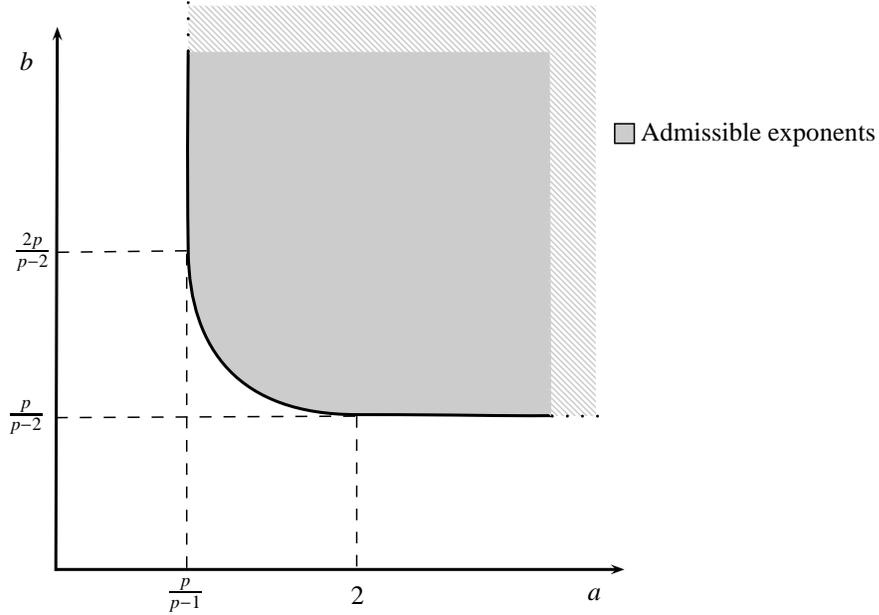
\begin{figure}[h]
\scalebox{1} 
{ \begin{pspicture}(0,-4.078848)(12.528457,4.0588474)
\definecolor{color1113b}{rgb}{0.8,0.8,0.8}
\psline[linewidth=0.04cm,arrowsize=0.05291667cm 2.0,arrowlength=1.4,arrowinset=0.4]{->}(1.5610156,-3.5411522)(1.5810156,3.6988478)
\psline[linewidth=0.04cm,arrowsize=0.05291667cm 2.0,arrowlength=1.4,arrowinset=0.4]{->}(1.5610156,-3.5211523)(9.121016,-3.5211523)
\usefont{T1}{ptm}{m}{n}
\rput(1.1724707,3.2438476){$b$}
\usefont{T1}{ptm}{m}{n}
\rput(8.72247,-3.8561523){$a$}
\psline[linewidth=0.02cm,linestyle=dashed,dash=0.16cm 0.16cm](1.5610156,-1.5011524)(5.5410156,-1.4811523)
\psline[linewidth=0.02cm,linestyle=dashed,dash=0.16cm 0.16cm](1.5810156,0.69884765)(3.3010156,0.71884763)
\psbezier[linewidth=0.04,fillstyle=solid,fillcolor=color1113b](3.3210156,3.3788476)(3.3210156,3.2388477)(3.3010156,1.9988476)(3.3210156,0.69884765)(3.3410156,-0.60115236)(4.0010157,-1.4611523)(5.5610156,-1.4611523)(7.1210155,-1.4611523)(7.341016,-1.4811523)(8.1610155,-1.4746797)
\psline[linewidth=0.02cm,linestyle=dashed,dash=0.16cm 0.16cm](3.3010156,0.6788477)(3.3010156,-3.5011523)
\usefont{T1}{ptm}{m}{n}
\rput(5.5724707,-3.8561523){$2$}
\psline[linewidth=0.02cm,linestyle=dashed,dash=0.16cm 0.16cm](5.5610156,-1.4811523)(5.5610156,-3.4811523)
\usefont{T1}{ptm}{m}{n}
\rput(3.2824707,-3.8561523){$\frac{p}{p-1}$}
\usefont{T1}{ptm}{m}{n}
\rput(1.1624707,-1.4961524){$\frac{p}{p-2}$}
\usefont{T1}{ptm}{m}{n}
\rput(1.2524707,0.70384765){$\frac{2p}{p-2}$}
\psline[linewidth=0.002,linecolor=color1113b,fillstyle=solid,fillcolor=color1113b](3.3410156,3.3588476)(8.121016,3.3588476)(8.141016,-1.4611523)(3.3410156,3.3388476)(3.3410156,3.3388476)
\usefont{T1}{ptm}{m}{n}
\rput(10.89333,2.2638476){Admissible exponents}
\psframe[linewidth=0.02,dimen=outer,fillstyle=solid,fillcolor=color1113b](9.261016,2.4188476)(8.981015,2.1388476)
\psframe[linewidth=0.002,linecolor=color1113b,linestyle=dotted,dotsep=0.16cm,dimen=outer,fillstyle=vlines,hatchwidth=0.02,hatchangle=45.0,hatchsep=0.0332,hatchcolor=color1113b](8.741015,3.9788477)(3.3210156,3.3588476)
\psframe[linewidth=0.002,linecolor=white,linestyle=dotted,dotsep=0.16cm,dimen=outer,fillstyle=vlines,hatchwidth=0.02,hatchangle=45.0,hatchsep=0.0332,hatchcolor=color1113b](8.741015,3.3588476)(8.121016,-1.4811523)\psline[linewidth=0.04cm,fillcolor=color1113b,linestyle=dotted,dotsep=0.16cm](3.3210156,3.3588476)(3.3210156,4.0388474)
\psline[linewidth=0.04cm,fillcolor=color1113b,linestyle=dotted,dotsep=0.16cm](8.141016,-1.4811523)(8.781015,-1.4811523)
\end{pspicture}
}\caption{Plot of the region of all admissible anisotropic exponents $(a,b)$
for which the Hardy-Littlewood inequality remains universally bounded in the
case $p=q$. The curve joining the points $(\frac{p}{p-1}, \frac{2p}{p-2})$ and
$(2, \frac{p}{p-2})$ is the hyperbola $b=\frac{2p\cdot a}{(3p-4)\cdot a - 2p}%
$.}%
\end{figure}

\section{Dimension blow-up}

Theorem \ref{main} is a complete classification of all possible anisotropic
Hardy-Littlewood type estimates. In this Section we turn our attention to the
order in which the estimates blow-up with respect to the dimension in the
cases when the exponents $(a,b)$ are out of the range predicted by Theorem
\ref{main}. In this scenario, there is no such a constant $C$, independent of
dimension $n$, for which the inequality
\begin{equation}
\label{ineq-blow-upSect}\left(  \sum_{i=1}^{n}\left(  \sum_{j=1}^{n}\left\vert
T(e_{i},e_{j})\right\vert ^{a}\right)  ^{\frac{1}{a}\cdot b}\right)
^{\frac{1}{b}}\leq C \cdot\left\Vert T\right\Vert
\end{equation}
is satisfied universally for all bilinear operators $T\colon\ell_{p}^{n}%
\times\ell_{q}^{n}\rightarrow\mathbb{K}$. Our goal is to obtain the precise
dependence arising on $n$.

Hereafter in this Section, we denote a non Hardy--Littlewood pair of exponents
by $\left(  r_{1},r_{2}\right)  $ and divide the region occupied by the non
Hardy--Littlewood exponents $\left(  r_{1},r_{2}\right)  $ in four sub-regions:

\begin{itemize}
\item[(R1)] $\left(  r_{1},r_{2}\right)  $ such that $q^{\ast}\leq r_{1}\leq2$
and%
\[
\frac{1}{r_{1}}+\frac{1}{r_{2}}>\frac{3}{2}-\left(  \frac{1}{p}+\frac{1}%
{q}\right)  .
\]

\item[(R2)] $\left(  r_{1},r_{2}\right)  $ such that $r_{1}<q^{\ast}$ and
$r_{2}<\frac{2p}{p-2}.$

\item[(R3)] $\left(  r_{1},r_{2}\right)  $ such that $r_{1}<q^{\ast}$ and
$r_{2}>\frac{2p}{p-2}.$

\item[(R4)] $\left(  r_{1},r_{2}\right)  $  such that $r_{1}>2$ and
$r_{2}<\frac{pq}{pq-p-q}.$
\end{itemize}

We will also use the following notation:
\[
\mathscr{B}^{n}_{p,q} := \left\{  T\colon\ell_{p}^{n}\times\ell_{q}%
^{n}\rightarrow\mathbb{K} \ : \ T \text{ is a bilinear form and } \|T\| = 1
\right\}  .
\]

\begin{proposition}
\label{pacocamento-rate1} Let $p\in\left(  2,\infty\right]  $, $q\in
\lbrack2,\infty]$.

\begin{itemize}
\item[(i)] If $\left(  r_{1},r_{2}\right)  $ are in (R1) or in (R2), then
\[
\sup\limits_{\mathscr{B}^{n}_{p,q}} \ \left(  \sum_{i=1}^{n}\left(  \sum
_{j=1}^{n}\left\vert T(e_{i},e_{j})\right\vert ^{r_{1}}\right)  ^{\frac
{1}{r_{1}}\cdot r_{2}}\right)  ^{\frac{1}{r_{2}}} = \operatorname*{O} \left(
n^{\frac{1}{r_{1}}+\frac{1}{r_{2}}-\left(  \frac{3}{2}-\left(  \frac{1}%
{p}+\frac{1}{q}\right)  \right)  } \right)  ,
\]
and the exponent $\frac{1}{r_{1}}+\frac{1}{r_{2}}-\left(  \frac{3}{2}-\left(
\frac{1}{p}+\frac{1}{q}\right)  \right)  $ is optimal.

\item[(ii)] If $\left(  r_{1},r_{2}\right)  $ are in (R3), then
\[
\sup\limits_{\mathscr{B}^{n}_{p,q}} \ \left(  \sum_{i=1}^{n}\left(  \sum
_{j=1}^{n}\left\vert T(e_{i},e_{j} )\right\vert ^{r_{1}}\right)  ^{\frac
{1}{r_{1}}\cdot r_{2}}\right)  ^{\frac{1}{r_{2}}} = \operatorname*{O} \left(
n^{\frac{1}{r_{1}}-\frac{1}{q^{\ast}}} \right)  ,
\]
and the exponent $\frac{1}{r_{1}}-\frac{1}{q^{\ast}}$ is optimal.

\item[(iii)] If $\left(  r_{1},r_{2}\right)  $ are in (R4)$,$ then
\[
\sup\limits_{\mathscr{B}^{n}_{p,q}} \ \left(  \sum_{i=1}^{n}\left(  \sum
_{j=1}^{n}\left\vert T(e_{i},e_{j})\right\vert ^{r_{1}}\right)  ^{\frac
{1}{r_{1}}\cdot r_{2}}\right)  ^{\frac{1}{r_{2}}} = \operatorname*{O} \left(
n^{\frac{1}{r_{2}}-\frac{pq-p-q}{pq}} \right)  ,
\]
and the exponent $\frac{1}{r_{2}}-\frac{pq-p-q}{pq}$ is optimal.
\end{itemize}
\end{proposition}

\begin{proof}
(i) Suppose that $\left(  r_{1},r_{2}\right)  $ are in (R1) and let $\delta>0$ be
such that%
\[
\frac{1}{r_{1}}+\frac{1}{\delta}=\frac{3}{2}-\left(  \frac{1}{p}+\frac{1}%
{q}\right)  .
\]
By H\"{o}lder inequality for mixed sums and Theorem \ref{main}, if%
\[
\frac{1}{X}+\frac{1}{\delta}=\frac{1}{r_{2}},
\]
then%
\begin{align*}
\left(  \sum_{i=1}^{n}\left(  \sum_{j=1}^{n}\left\vert T(e_{i}%
,e_{j})\right\vert ^{r_{1}}\right)  ^{\frac{1}{r_{1}}\cdot r_{2}}\right)
^{\frac{1}{r_{2}}}
&  \leq \left(  \sum_{i=1}^{n}\left(  \sum_{j=1}^{n}\left\vert T(e_{i}%
,e_{j})\right\vert ^{r_{1}}\right)  ^{\frac{1}{r_{1}}\cdot\delta}\right)
^{\frac{1}{\delta}}\cdot\left(  \sum_{i=1}^{n}\left(  \sup_{j}1\right)
^{X}\right)  ^{\frac{1}{X}}\\
&  \leq C\left\Vert T\right\Vert n^{\frac{1}{X}}\\
&  =C\left\Vert T\right\Vert n^{\frac{1}{r_{1}}+\frac{1}{r_{2}}-\left(
\frac{3}{2}-\left(  \frac{1}{p}+\frac{1}{q}\right)  \right)  }.
\end{align*}
To prove the optimality, consider%
\[
A_{n} \colon \ell_{p}\times\ell_{q}\rightarrow\mathbb{K}%
\]
given by the Kahane--Salem--Zygmund inequality. If%
\[
\left(  \sum_{i=1}^{n}\left(  \sum_{j=1}^{n}\left\vert A_{n}(e_{i}%
,e_{j})\right\vert ^{r_{1}}\right)  ^{\frac{1}{r_{1}}\cdot r_{2}}\right)
^{\frac{1}{r_{2}}}\leq Cn^{t}\left\Vert A_{n}\right\Vert
\]
for a certain $t$, then%
\[
n^{\frac{1}{r_{1}}+\frac{1}{r_{2}}}\leq Cn^{t}n^{\frac{3}{2}-\left(  \frac
{1}{p}+\frac{1}{q}\right)  }.
\]
Since $n$ is arbitrary, we conclude that%
\[
t\geq\frac{1}{r_{1}}+\frac{1}{r_{2}}-\left(  \frac{3}{2}-\left(  \frac{1}%
{p}+\frac{1}{q}\right)  \right).
\]
Let us now consider the case when  $\left(  r_{1},r_{2}\right)  $ lies in the region (R2). Let
$X_{1},X_{2}$ be such that%
\begin{align*}
\frac{1}{r_{1}} &  =\frac{1}{q^{\ast}}+\frac{1}{X_{1}},\\
\frac{1}{r_{2}} &  =\frac{1}{\frac{2p}{p-2}}+\frac{1}{X_{2}}.
\end{align*}
By the H\"{o}lder inequality for mixed sums and Theorem \ref{main}, we can estimate
\begin{align*}
\left(  \sum_{i=1}^{n}\left(  \sum_{j=1}^{n}\left\vert T(e_{i}%
,e_{j})\right\vert ^{r_{1}}\right)  ^{\frac{1}{r_{1}}\cdot r_{2}}\right)
^{\frac{1}{r_{2}}}
&  \leq\left(  \sum_{i=1}^{n}\left(  \sum_{j=1}^{n}\left\vert T(e_{i}%
,e_{j})\right\vert ^{q^{\ast}}\right)  ^{\frac{1}{q^{\ast}}\cdot\frac
{2p}{p-2}}\right)  ^{\frac{p-2}{2p}}\cdot\left(  \sum_{i=1}^{n}\left(
\sum_{j=1}^{n}\left\vert 1\right\vert ^{X_{1}}\right)  ^{\frac{1}{X_{1}}\cdot
X_{2}}\right)  ^{\frac{1}{X_{2}}}\\
&  \leq C\left\Vert T\right\Vert n^{\frac{1}{r_{1}}+\frac{1}{r_{2}}-\frac
{1}{q^{\ast}}-\frac{1}{\frac{2p}{p-2}}}\\
&  =C\left\Vert T\right\Vert n^{\frac{1}{r_{1}}+\frac{1}{r_{2}}-\left(
\frac{3}{2}-\left(  \frac{1}{p}+\frac{1}{q}\right)  \right)  }.
\end{align*}
To prove the optimality, consider%
\[
A_{n} \colon \ell_{p}\times\ell_{q}\rightarrow\mathbb{K}%
\]
given by the Kahane--Salem--Zygmund inequality. If%
\[
\left(  \sum_{i=1}^{n}\left(  \sum_{j=1}^{n}\left\vert A_{n}(e_{i}%
,e_{j})\right\vert ^{r_{1}}\right)  ^{\frac{1}{r_{1}}\cdot r_{2}}\right)
^{\frac{1}{r_{2}}}\leq Cn^{t}\left\Vert  A_{n}\right\Vert ,
\]
for a certain $t,$ then%
\[
n^{\frac{1}{r_{1}}+\frac{1}{r_{2}}}\leq Cn^{t}n^{\frac{3}{2}-\left(  \frac
{1}{p}+\frac{1}{q}\right)  }.
\]
Since $n$ is arbitrary, we conclude that%
\[
t\geq\frac{1}{r_{1}}+\frac{1}{r_{2}}-\left(  \frac{3}{2}-\left(  \frac{1}%
{p}+\frac{1}{q}\right)  \right)  .
\]
\noindent (ii) By the H\"{o}lder inequality for mixed sums and Theorem \ref{main}, if%
\[
\frac{1}{q^{\ast}}+\frac{1}{\delta}=\frac{1}{r_{1}},
\]
then%
\begin{align*}
\left(  \sum_{i=1}^{n}\left(  \sum_{j=1}^{n}\left\vert T(e_{i}%
,e_{j})\right\vert ^{r_{1}}\right)  ^{\frac{1}{r_{1}}\cdot r_{2}}\right)
^{\frac{1}{r_{2}}}
&  \leq\left(  \sum_{i=1}^{n}\left(  \sum_{j=1}^{n}\left\vert T(e_{i}%
,e_{j})\right\vert ^{q^{\ast}}\right)  ^{\frac{1}{q^{\ast}}\cdot r_{2}%
}\right)  ^{\frac{1}{r_{2}}}\cdot \sup_{i}\left(  \sum_{j=1}%
^{n}1^{\delta}\right)  ^{\frac{1}{\delta}}\\
&  \leq C\left\Vert T\right\Vert n^{\frac{1}{r_{1}}-\frac{1}{q^{\ast}}}.
\end{align*}
To prove the optimality, consider%
\[
A_{n} \colon \ell_{p}\times\ell_{q}\rightarrow\mathbb{K}%
\]
given by
\[
A_{n}(x,y)=x_{1}%
{\textstyle\sum\limits_{j=1}^{n}}
y_{j}.
\]
If%
\[
\left(  \sum_{i=1}^{n}\left(  \sum_{j=1}^{n}\left\vert A_{n}(e_{i}%
,e_{j})\right\vert ^{r_{1}}\right)  ^{\frac{1}{r_{1}}\cdot r_{2}}\right)
^{\frac{1}{r_{2}}}\leq Cn^{t}\left\Vert A_{n}\right\Vert ,
\]
for a certain $t,$ then%
\[
n^{\frac{1}{r_{1}}}\leq Cn^{t}n^{\frac{1}{q^{\ast}}}.
\]
Since $n$ is arbitrary, we conclude that%
\[
t\geq\frac{1}{r_{1}}-\frac{1}{q^{\ast}}.
\]
\noindent (iii) By the H\"{o}lder inequality for mixed sums and Theorem \ref{main}, if%
\[
\frac{1}{\frac{pq}{pq-p-q}}+\frac{1}{\delta}=\frac{1}{r_{2}},
\]
then%
\begin{align*}
\left(  \sum_{i=1}^{n}\left(  \sum_{j=1}^{n}\left\vert T(e_{i}%
,e_{j})\right\vert ^{r_{1}}\right)  ^{\frac{1}{r_{1}}\cdot r_{2}}\right)
^{\frac{1}{r_{2}}}
&  \leq\left(  \sum_{i=1}^{n}\left(  \sum_{j=1}^{n}\left\vert T(e_{i}%
,e_{j})\right\vert ^{r_{1}}\right)  ^{\frac{1}{r_{1}}\cdot\frac{pq}{pq-p-q}%
}\right)  ^{\frac{1}{\frac{pq}{pq-p-q}}}\cdot \left(  \sum_{i=1}^{n}\left ( \sup_{j}1^{\delta
} \right ) \right)  ^{\frac{1}{\delta}}\\
&  \leq C\left\Vert T\right\Vert n^{\frac{1}{r_{2}}-\frac{pq-p-q}{pq}}.
\end{align*}
To prove the optimality, consider%
\[
A_{n} \colon \ell_{p}\times\ell_{q}\rightarrow\mathbb{K}%
\]
given by
\[
A_{n}(x,y)=%
{\textstyle\sum_{j=1}^{n}}
x_{j}y_{j}.
\]
If%
\[
\left(  \sum_{i=1}^{n}\left(  \sum_{j=1}^{n}\left\vert A_{n}(e_{i}%
,e_{j})\right\vert ^{r_{1}}\right)  ^{\frac{1}{r_{1}}\cdot r_{2}}\right)
^{\frac{1}{r_{2}}}\leq Cn^{t}\left\Vert A_{n}\right\Vert
\]
for a certain $t,$ then%
\[
n^{\frac{1}{r_{2}}}\leq Cn^{t}n^{1-\left(  \frac{1}{p}+\frac{1}{q}\right)  }.
\]
Since $n$ is arbitrary, we conclude that%
\[
t\geq\frac{1}{r_{2}}-\frac{pq-p-q}{pq},
\]
and the proposition is proven.
\end{proof}

The case $p=2$ involves a simpler analysis.

\begin{proposition}
\label{pacocamento-rate2} Let $q\in(2,\infty]$ and $\left(  r_{1}%
,r_{2}\right)  $ be not a pair of Hardy--Littlewood exponents, then
\[
\sup\limits_{\mathscr{B}^{n}_{2,q}} \ \left(  \sum_{i=1}^{n}\left(  \sum
_{j=1}^{n}\left\vert T(e_{i},e_{j})\right\vert ^{r_{1}}\right)  ^{\frac
{1}{r_{1}}\cdot r_{2}}\right)  ^{\frac{1}{r_{2}}}\leq\operatorname*{O} \left(
n^{\frac{1}{r_{1}}+\frac{1}{r_{2}}-\frac{1}{q^{\ast}}}\right)  ,
\]
and the exponent $\frac{1}{r_{1}}+\frac{1}{r_{2}}-\frac{1}{q^\ast}$ is optimal.
\end{proposition}

\begin{proof}
If $r_{1}<q^{\ast},$  we set $X$ to satisfy
\[
\frac{1}{X}+\frac{1}{q^{\ast}}=\frac{1}{r_{1}}.
\]
From H\"{o}lder inequality for mixed sums and (\ref{mmmm8}), we reach
\begin{align*}
\left(  \sum_{i=1}^{n}\left(  \sum_{j=1}^{n}\left\vert T(e_{i}%
,e_{j})\right\vert ^{r_{1}}\right)  ^{\frac{1}{r_{1}}\cdot r_{2}}\right)
^{\frac{1}{r_{2}}}
&  \leq\sup_{i}\left(  \sum_{j=1}^{n}\left\vert T(e_{i},e_{j})\right\vert
^{q^{\ast}}\right)  ^{\frac{1}{q^{\ast}}}\cdot\left(  \sum_{i=1}^{n}\left(
\sum_{j=1}^{n}1^{X}\right)  ^{\frac{1}{X}\cdot r_{2}}\right)  ^{\frac
{1}{r_{2}}}\\
&  \leq C_{0}\left\Vert T\right\Vert n^{\frac{1}{r_{1}}+\frac{1}{r_{2}}-\frac
{1}{q^{\ast}}}.
\end{align*}
If $r_{1}\geq q^{\ast},$ we set $X$ and $Y$ through the equalities
\begin{align*}
\frac{1}{r_{1}}+\frac{1}{X}  &  =\frac{1}{q^{\ast}},\\
\frac{1}{Y}+\frac{1}{X}  &  =\frac{1}{r_{2}}.%
\end{align*}
Applying once more H\"{o}lder inequality for mixed sums and Theorem \ref{main}, we obtain a constant $C$ such that
\begin{align*}
\left(  \sum_{i=1}^{n}\left(  \sum_{j=1}^{n}\left\vert T(e_{i}%
,e_{j})\right\vert ^{r_{1}}\right)  ^{\frac{1}{r_{1}}\cdot r_{2}}\right)
^{\frac{1}{r_{2}}}
&  \leq\left(  \sum_{i=1}^{n}\left(  \sum_{j=1}^{n}\left\vert T(e_{i}%
,e_{j})\right\vert ^{r_{1}}\right)  ^{\frac{1}{r_{1}}\cdot X}\right)
^{\frac{1}{X}}\cdot \left(\sum_{i=1}^{n}  \left(\sup_{j}1\right)^{Y}\right)  ^{\frac{1}{Y}%
}\\
&  \leq C\left\Vert T\right\Vert n^{\frac{1}{r_{1}}+\frac{1}{r_{2}}-\frac
{1}{q^{\ast}}}.
\end{align*}
The optimality is proved as in Proposition \ref{pacocamento-rate1}
\end{proof}

{When $p,q>1$ are such that
\[
\frac{1}{p}+\frac{1}{q}\geq1,
\]
it is easy to verify that there is no admissible exponent for which
\eqref{ineq-blow-upSect} remains universally bounded. Next, we investigate the
dimensional dependence in this case and, for that, it is useful to denote
\[
\mathscr{B}_{p,q}^{n_{1},n_{2}}:=\left\{  T\colon\ell_{p}^{n_{1}}\times
\ell_{q}^{n_{2}}\rightarrow\mathbb{K}\ :\ T\text{ is a bilinear form and
}\Vert T\Vert=1\right\}  .
\]
}

\begin{proposition}
Let $p,q>1$ be such that $\frac{1}{p}+\frac{1}{q}\geq1$. There holds
\[
\sup\limits_{\mathscr{B}_{p,q}^{n_{1},n_{2}}}\ \left(  \sum_{i=1}^{n_{1}%
}\left(  \sum_{j=1}^{n_{2}}\left\vert T(e_{i},e_{j})\right\vert ^{r_{1}%
}\right)  ^{\frac{1}{r_{1}}\cdot r_{2}}\right)  ^{\frac{1}{r_{2}}%
}=\operatorname*{O}\left(  n_{1}^{\frac{1}{r_{2}}}\right)  \cdot
\operatorname*{O}\left(  n_{2}^{\frac{1}{r_{1}}-\frac{1}{\max\left\{
r_{1},q^{\ast}\right\}  }}\right)  ,
\]
and such a blow-up rate of $n_{2}$ is sharp and the blow up rate of $n_{1}$ is
sharp when $r_{1}\geq q^{\ast}$.

\end{proposition}

\begin{proof}
Under the condition $\frac{1}{p}+\frac{1}{q}\geq1$, it is simple to check that
any $T\in\mathscr{B}_{p,q}^{n_{1},n_{2}}$ is multiple $\left(  \infty
,\max\left\{  r_{1},q^{\ast}\right\}  ;p^{\ast},q^{\ast}\right)  $-summing.
Let us denote $T(e_{i},e_{j})=T_{ij}$. By Hölder inequality for mixed sums, if%
\[
\frac{1}{r_{1}}=\frac{1}{t_{1}}+\frac{1}{\max\left\{  r_{1},q^{\ast}\right\}
},
\]
then%
\begin{align*}
\left(  \sum_{i=1}^{n_{1}}\left(  \sum_{j=1}^{n_{2}}\left\vert T_{ij}%
\right\vert ^{r_{1}}\right)  ^{\frac{1}{r_{1}}\cdot r_{2}}\right)  ^{\frac
{1}{r_{2}}} &  \leq\left(  \sum_{i=1}^{n_{1}}\left(  \sum_{j=1}^{n_{2}%
}\left\vert 1\right\vert ^{t_{1}}\right)  ^{\frac{1}{t_{1}}\cdot r_{2}%
}\right)  ^{\frac{1}{r_{2}}}\cdot\sup_{i}\left(  \sum_{j=1}^{n_{2}}\left\vert
T_{ij}\right\vert ^{\max\left\{  r_{1},q^{\ast}\right\}  }\right)  ^{\frac
{1}{\max\left\{  r_{1},q^{\ast}\right\}  }}\\
&  \leq C\left\Vert T\right\Vert n_{1}^{\frac{1}{r_{2}}}n_{2}^{\frac{1}{r_{1}%
}-\frac{1}{\max\left\{  r_{1},q^{\ast}\right\}  }}.
\end{align*}
To prove the optimality of the exponent $\frac{1}{r_{1}}-\frac{1}{\max\left\{
r_{1},q^{\ast}\right\}  }$ we just need to consider $r_{1}\leq q^{\ast};$ let
us consider%
\[
A_{n_{2}}\colon\ell_{p}\times\ell_{q}\rightarrow\mathbb{K}%
\]
given by
\[
A_{n_{2}}(x,y)=x_{1}%
{\textstyle\sum\limits_{j=1}^{n_{2}}}
y_{j}.
\]
If%
\[
\left(  \sum_{i=1}^{n_{1}}\left(  \sum_{j=1}^{n_{2}}\left\vert A_{n_{2}}%
(e_{i},e_{j})\right\vert ^{r_{1}}\right)  ^{\frac{1}{r_{1}}\cdot r_{2}%
}\right)  ^{\frac{1}{r_{2}}}\leq Cn_{1}^{t_{1}}n_{2}^{t_{2}}\left\Vert
A_{n_{2}}\right\Vert ,
\]
for certain $t_{1},t_{2},$ then%
\[
n_{2}^{\frac{1}{r_{1}}}\leq Cn_{1}^{t_{1}}n_{2}^{t_{2}}n_{2}^{\frac{1}%
{q^{\ast}}}.
\]
Since $n_{2}$ is arbitrary, we conclude that%
\[
t_{2}\geq\frac{1}{r_{1}}-\frac{1}{q^{\ast}}=\frac{1}{r_{1}}-\frac{1}%
{\max\left\{  r_{1},q^{\ast}\right\}  }.
\]
Now let us prove the optimality of the exponent $\frac{1}{r_{2}}$ of $n_{1}$ when
$r_{1}\geq q^{\ast}.$ Consider
\[
A_{n}\colon\ell_{p}\times\ell_{q}\rightarrow\mathbb{K}%
\]
given by
\[
A_{n}(x,y)=%
{\textstyle\sum\limits_{j=1}^{n}}
x_{j}y_{j}.
\]
Since $\left\Vert A_{n}\right\Vert =1$, if%
\[
\left(  \sum_{i=1}^{n_{1}}\left(  \sum_{j=1}^{n_{2}}\left\vert A_{n}%
(e_{i},e_{j})\right\vert ^{r_{1}}\right)  ^{\frac{1}{r_{1}}\cdot r_{2}%
}\right)  ^{\frac{1}{r_{2}}}\leq Cn_{1}^{t_{1}}\left\Vert A_{n}\right\Vert ,
\]
for a certain $t_{1}$ and all $n_{1},n_{2}$, then considering $n_{1}=n_{2}=n$
we have
\[
n^{\frac{1}{r_{2}}}\leq Cn^{t_{1}}.
\]
Since $n$ is arbitrary, we find%
\[
t_{1}\geq\frac{1}{r_{2}},
\]
which concludes the proof.
\end{proof}

\section{Remarks on the multilinear case \label{mul}}

In this final Section we comment on the sharp exponent problem for multilinear
versions of the Hardy-Littlewood inequality, in the spirit of \cite{pra}.
Quite recently Dimant and Sevilla-Peris established the existence of a
constant $C_{m,p}\geq1$, such that
\begin{equation}
\left(  \sum\limits_{i_{1},\ldots,i_{m}=1}^{\infty}\left\vert T(e_{i_{^{1}}%
},\ldots,e_{i_{m}})\right\vert ^{\frac{p}{p-m}}\right)  ^{\frac{p-m}{p}}\leq
C_{m,p}\left\Vert T\right\Vert \label{6kk}%
\end{equation}
for all $m$-linear operators $T \colon\ell_{p}^{n}\times\cdots\times\ell
_{p}^{n}\rightarrow\mathbb{K}$, with $m\ge1$ and $p\in(m,2m).$ Furthermore,
they have also shown that the exponent
\[
\mathfrak{e} := \frac{p}{p-m}
\]
is sharp in the sense that it cannot be replaced by any $a<\frac{p}{p-m}$ in
(\ref{6kk}).

As the condition $p<2m$ is in order, it follows readily that
\[
\frac{1}{\frac{p}{p-m}}+\cdots+\frac{1}{\frac{p}{p-m}}=\frac{m\left(
p-m\right)  }{p}<\frac{m+1}{2}-\frac{m}{p},
\]
which, having in mind the Kahane--Salem--Zygmund inequality, indicates that
the optimal exponents $\frac{p}{p-m}$ seem to be sub-optimal in the
anisotropic setting. The main result of this section shows, in particular,
that in fact the optimal exponents $\frac{p}{p-m}$ are not optimal in the
anisotropic stronger sense.

Next definition seems to play a decisive role in the sharp exponent problem
for multilinear operators.

\begin{definition}
\label{Def.GS} An $m$-uple of exponents $\left(  q_{1},...,q_{m}\right)  $ for
which a Hardy--Littlewood type inequality holds and that for any
$\varepsilon_{j}>0$ and any $j=1,...,m$, there is no Hardy--Littlewood
inequality for the $m$-uple of exponents
\[
\left(  q_{1},...,q_{j-1} ,q_{j}-\varepsilon_{j},q_{j+1},...,q_{m}\right)
\]
is called globally sharp.
\end{definition}

\bigskip

A careful application of the tools and reasoning developed in Section
\ref{Sct RP}, combined with techniques from the theory of absolutely summing
operators, \cite{lindenstrauss}, yields the following result, that extends the
reach of (\ref{6kk}) and Theorem \ref{BH_HL} with globally sharp exponents:

\begin{theorem}
\label{qq2211} Let $m\ge 3$ be a positive integer and $p\geq2m-2$ be a real number.
Then there is a constant $C_{m,p}\geq1$ such that%
\[
\left(  \sum_{i_{1}=1}^{n}\left(  \sum_{i_{2},...,i_{m}=1}^{n}\left\vert
T(e_{i_{1}},...,e_{i_{m}})\right\vert ^{\frac{2\left(  m-1\right)  p}%
{mp-2m+2}}\right)  ^{\frac{mp-2m+2}{2\left(  m-1\right)  p}\times\frac
{2p}{p-2}}\right)  ^{\frac{p-2}{2p}}\leq C_{m,p}\left\Vert T\right\Vert
\]
for all $m$-linear operators $T:\ell_{p}^{n}\times\cdots\times\ell_{p}%
^{n}\rightarrow\mathbb{K}$ and all positive integers $n$. Moreover the
$m$-uple of exponents $\left(  \frac{2p}{p-2},\frac{2\left(  m-1\right)
p}{mp-2m+2},\overset{m-1\text{ times}}{...},\frac{2\left(  m-1\right)
p}{mp-2m+2}\right)  $ is globally sharp.
\end{theorem}

\begin{proof}
Let $T:\ell_{p}\times\cdots\times\ell_{p}\rightarrow\mathbb{K}$ be a
continuous $m$-linear form. By the Khinchine inequality, there is a constant
$K_{m,p}\geq1$ such that
\begin{align*}
&  \left(  \sum\limits_{i_{2},\ldots,i_{m}=1}^{n}\left(  \sum\limits_{i_{1}%
=1}^{n}\left\vert T(x_{i_{^{1}}}^{(1)},\ldots,x_{i_{m}}^{(m)})\right\vert
^{2}\right)  ^{\frac{1}{2}\times\frac{2\left(  m-1\right)  p}{\left(
m-1\right)  p+p-2\left(  m-1\right)  }}\right)  ^{\frac{\left(  m-1\right)
p+p-2\left(  m-1\right)  }{2\left(  m-1\right)  p}}\\
&  \leq K_{m,p}\left(  \sum\limits_{i_{2},\ldots,i_{m}=1}^{n}\int
\limits_{0}^{1}\left\vert \sum\limits_{i_{1}=1}^{n}r_{i_{1}}(t)T(x_{i_{^{1}}%
}^{(1)},\ldots,x_{i_{m-1}}^{(m-1)},x_{i_{m}}^{(m)})\right\vert ^{\frac
{2\left(  m-1\right)  p}{\left(  m-1\right)  p+p-2\left(  m-1\right)  }%
}dt\right)  ^{\frac{\left(  m-1\right)  p+p-2\left(  m-1\right)  }{2\left(
m-1\right)  p}}\\
&  =K_{m,p}\left(  \int\limits_{0}^{1}\sum\limits_{i_{2},\ldots,i_{m}=1}%
^{n}\left\vert T\left(  \sum\limits_{i_{1}=1}^{n}r_{i_{1}}(t)x_{i_{^{1}}%
}^{(1)},\ldots,x_{i_{m-1}}^{(m-1)},x_{i_{m}}^{(m)}\right)  \right\vert
^{\frac{2\left(  m-1\right)  p}{\left(  m-1\right)  p+p-2\left(  m-1\right)
}}dt\right)  ^{\frac{\left(  m-1\right)  p+p-2\left(  m-1\right)  }{2\left(
m-1\right)  p}}.
\end{align*}
Combining the previous inequality with Theorem \ref{BH_HL} for the $\left(
m-1\right)  $-linear operator $T(\sum\limits_{i_{1}=1}^{n}r_{i_{1}%
}(t)x_{i_{1}}^{(1)},\cdot,\ldots,\cdot)$ we conclude that there is a constant
$L_{m,p}\geq1$ such that%
\begin{align*}
&  \left(  \sum\limits_{i_{2},\ldots,i_{m}=1}^{n}\left(  \sum\limits_{i_{1}%
=1}^{n}\left\vert T(x_{i_{^{1}}}^{(1)},\ldots,x_{i_{m}}^{(m)})\right\vert
^{2}\right)  ^{\frac{1}{2}\times\frac{2\left(  m-1\right)  p}{\left(
m-1\right)  p+p-2\left(  m-1\right)  }}\right)  ^{\frac{\left(  m-1\right)
p+p-2\left(  m-1\right)  }{2\left(  m-1\right)  p}}\\
&  \leq L_{m,p}\sup_{t\in\lbrack0,1]}\left\Vert T(\sum\limits_{i_{1}%
=1}^{n}r_{i_{1}}(t)x_{i_{1}}^{(1)},\cdot,\ldots,\cdot)\right\Vert
{\textstyle\prod\limits_{j=2}^{m}}
\left\Vert \left(  x_{i_{j}}^{(j)}\right)  _{i_{j}=1}^{n}\right\Vert
_{w,p^{\ast}}\\
&  \leq L_{m,p}\left\Vert T\right\Vert \sup_{t\in\lbrack0,1]}\left\Vert
\sum\limits_{i_{1}=1}^{n}r_{i_{1}}(t)x_{i_{1}}^{(1)}\right\Vert
{\textstyle\prod\limits_{j=2}^{m}}
\left\Vert \left(  x_{i_{j}}^{(j)}\right)  _{i_{j}=1}^{n}\right\Vert
_{w,p^{\ast}}.
\end{align*}
Since (see \cite[page 284]{lindenstrauss})
\[
\left\Vert \left(  x_{i_{1}}^{(1)}\right)  _{i_{1}=1}^{n}\right\Vert
_{w,1}=\max\left\Vert \sum\limits_{i_{1}=1}^{n}\varepsilon_{i_{1}}x_{i_{1}%
}^{(1)}:\varepsilon_{i_{1}}=\pm1,i=1,...,n\right\Vert ,
\]
we have%
\begin{align*}
&  \left(  \sum\limits_{i_{2},\ldots,i_{m}=1}^{n}\left(  \sum\limits_{i_{1}%
=1}^{n}\left\vert T(x_{i_{^{1}}}^{(1)},\ldots,x_{i_{m}}^{(m)})\right\vert
^{2}\right)  ^{\frac{1}{2}\times\frac{2\left(  m-1\right)  }{\left(
m-1\right)  p+p-2\left(  m-1\right)  }}\right)  ^{\frac{\left(  m-1\right)
p+p-2\left(  m-1\right)  }{2\left(  m-1\right)  }}\\
&  \leq L_{m,p}\left\Vert T\right\Vert \left\Vert \left(  x_{i_{1}}%
^{(1)}\right)  _{i_{1}=1}^{n}\right\Vert _{w,1}%
{\textstyle\prod\limits_{j=2}^{m}}
\left\Vert \left(  x_{i_{j}}^{(j)}\right)  _{i_{j}=1}^{n}\right\Vert
_{w,p^{\ast}}.
\end{align*}
Since $p\geq2m-2,$ by the Minkowski inequality (\ref{7j}) we conclude that%
\begin{align*}
&  \left(  \sum\limits_{i_{1}=1}^{n}\left(  \sum\limits_{i_{2},\ldots,i_{m}%
=1}^{n}\left\vert T(x_{i_{^{1}}}^{(1)},\ldots,x_{i_{m}}^{(m)})\right\vert
^{\frac{2\left(  m-1\right)  p}{\left(  m-1\right)  p+p-2\left(  m-1\right)
}}\right)  ^{\frac{\left(  m-1\right)  p+p-2\left(  m-1\right)  }{2\left(
m-1\right)  p}\times2}\right)  ^{\frac{1}{2}}\\
&  \leq L_{m,p}\left\Vert T\right\Vert \left\Vert \left(  x_{i_{1}}%
^{(1)}\right)  _{i_{1}=1}^{n}\right\Vert _{w,1}%
{\textstyle\prod\limits_{j=2}^{m}}
\left\Vert \left(  x_{i_{j}}^{(j)}\right)  _{i_{j}=1}^{n}\right\Vert
_{w,p^{\ast}}.
\end{align*}
Thus, $T$ is multiple $\left(  2,\frac{2\left(  m-1\right)  p}{mp-2\left(
m-1\right)  },...,\frac{2\left(  m-1\right)  p}{mp-2\left(  m-1\right)
};1,p^{\ast},...,p^{\ast}\right)  $-summing. By the Anisotropic Regularity Principle we conclude
that%
\begin{align*}
&  \left(  \sum\limits_{i_{1}=1}^{n}\left(  \sum\limits_{i_{2},\ldots,i_{m}%
=1}^{n}\left\vert T(x_{i_{^{1}}}^{(1)},\ldots,x_{i_{m}}^{(m)})\right\vert
^{\frac{2\left(  m-1\right)  p}{\left(  m-1\right)  p+p-2\left(  m-1\right)
}}\right)  ^{\frac{\left(  m-1\right)  p+p-2\left(  m-1\right)  }{2\left(
m-1\right)  p}\times\frac{2p}{p-2}}\right)  ^{\frac{p-2}{2p}}\\
&  \leq L_{m,p}\left\Vert T\right\Vert
{\textstyle\prod\limits_{j=1}^{m}}
\left\Vert \left(  x_{i_{j}}^{(j)}\right)  _{i_{j}=1}^{n}\right\Vert
_{w,p^{\ast}}%
\end{align*}
and, recalling Theorem \ref{yh}, the proof is done.
Since%
\[
\frac{1}{\frac{2\left(  m-1\right)  p}{\left(  m-1\right)  p+p-2\left(
m-1\right)  }}+\overset{m-1\text{ times}}{...}+\frac{1}{\frac{2\left(
m-1\right)  p}{\left(  m-1\right)  p+p-2\left(  m-1\right)  }}+\frac{1}%
{\frac{2p}{p-2}}=\frac{m+1}{2}-\frac{m}{p},
\]
by the Kahane--Salem--Zygmund we conclude that $\left(  \frac{2p}{p-2}%
,\frac{2\left(  m-1\right)  p}{mp-2m+2},\overset{m-1\text{ times}}{...}%
,\frac{2\left(  m-1\right)  p}{mp-2m+2}\right)  $ is globally sharp.
\end{proof}

\medskip

It is worth to note that Theorem \ref{qq2211} provides new $m$-uples of
optimal exponents for multilinear Hardy--Littlewood type inequalities that
were not covered by the preceding literature.


\bigskip

\noindent\textsc{Daniel M. Pellegrino} \hfill\textsc{Joedson Santos}\newline
Department of Mathematics \hfill Department of Mathematics \newline
Universidade Federal da Para\'iba \hfill Universidade Federal da Para\'iba
\newline58.051-900 Jo\~ao Pessoa -- PB, Brazil. \hfill58.051-900 Jo\~ao Pessoa
-- PB, Brazil. \smallskip\newline\texttt{pellegrino@pq.cnpq.br} \hfill
\texttt{joedson@mat.ufpb.br}

\bigskip

\noindent\textsc{Diana Serrano-Rodr\'iguez} \hfill\textsc{Eduardo V.
Teixeira}\newline Department of Mathematics \hfill Department of Mathematics
\newline Kent State University \hfill Universidade Federal do Cear\'a
\newline44242 Kent -- Ohio, USA \hfill60.455-760 Fortaleza -- CE, Brazil
\smallskip\newline\texttt{dserrano@kent.edu} \hfill
\texttt{teixeira@mat.ufc.br}



\end{document}